\documentclass[hidelinks,onefignum,onetabnum]{siamart220329}

\usepackage{lipsum}
\usepackage{amsfonts}
\usepackage{graphicx}
\usepackage{epstopdf}
\usepackage{algpseudocode}
\usepackage{subcaption}
\usepackage{hyperref}
\hypersetup{
    colorlinks=true,
    linkcolor=blue,
    filecolor=magenta,      
    urlcolor=cyan,
    pdftitle={Overleaf Example},
    pdfpagemode=FullScreen,
}

\ifpdf
  \DeclareGraphicsExtensions{.eps,.pdf,.png,.jpg}
\else
  \DeclareGraphicsExtensions{.eps}
\fi

\newsiamremark{remark}{Remark}
\newsiamremark{hypothesis}{Hypothesis}
\crefname{hypothesis}{Hypothesis}{Hypotheses}
\newsiamthm{claim}{Claim}

\headers{FIRM}{G. Byeon, M. Ryu, Z. Di, K. Kim}

\title{FIRM: Federated Image Reconstruction using Multimodal Tomographic Data\thanks{Submitted to the editors DATE.
\funding{This work was funded by the DOE.}}}

\author{Geunyeong Byeon\thanks{School of Computing and Augmented Intelligence, Arizona State University, Tempe, AZ \email{\{geunyeong.byeon, mryu2\}@asu.edu}.}
\and Minseok Ryu\footnotemark[2]
\and Zichao Wendy Di\thanks{Mathematics and Computer Science, Argonne National Laboratory, Lemont, IL \email{\{wendydi, kimk\}@anl.gov}.}
\and Kibaek Kim\footnotemark[3]
}

\usepackage{amsopn}

\newcommand{\wendy}[1]           
{\textcolor{red}{#1}}

\ifpdf
\hypersetup{
  pdftitle={Federated Image Reconstruction using Multimodal Tomographic Data},
  pdfauthor={Geunyeong Byeon, Minseok Ryu, Zichao Wendy Di, Kibaek Kim}
}
\fi

\usepackage{booktabs}
\usepackage{enumitem}
\DeclareMathOperator*{\argmin}{arg\,min}

\begin{document}

\maketitle

\begin{abstract}
    We propose a federated algorithm for reconstructing images using multimodal tomographic data sourced from dispersed locations, addressing the challenges of traditional unimodal approaches that are prone to noise and reduced image quality. Our approach formulates a joint inverse optimization problem incorporating multimodality constraints and solves it in a federated framework through local gradient computations complemented by lightweight central operations, ensuring data decentralization. Leveraging the connection between our federated algorithm and the quadratic penalty method, we introduce an adaptive step-size rule with guaranteed sublinear convergence and further suggest its extension to augmented Lagrangian framework. Numerical results demonstrate its superior computational efficiency and improved image reconstruction quality.
\end{abstract}

\vspace{-2mm}
\vspace{-2mm}
\vspace{-3mm}
\section{Introduction}
Tomographic imaging reconstructs objects from projections onto lower-dimensional spaces, obtained by scanning with penetrating waves from various angles, with diverse applications in fields such as biology, physics, and astronomy. However, achieving high-quality reconstruction is challenging due to the ill-posed nature of the underlying inverse problem, which is often formulated as a least squares problem. Insufficient measurements—caused by factors such as low-dose requirements in computed tomography (CT), restricted projection angles, limited field of view, and low signal-to-noise ratios—can result in non-unique solutions.

Different tomographic modalities, such as CT and magnetic resonance, have similar mathematical modeling and experimental challenges. Among these, X-ray transmission (XRT) tomography is widely used, while X-ray fluorescence (XRF) tomography has gained attention for its ability to trace the elemental composition of samples \cite{paunesku2006x,rust1998x,hogan_fluorescent_1991}. Combining these modalities shows promise for improving reconstruction quality \cite{di2016optimization,di2017joint}. However, centralizing multimodal data for the reconstruction is impractical or undesirable due to the large volumes of data generated by dispersed synchrotron facilities, which pose challenges for data transfer, storage, and raise privacy concerns. To overcome this, we propose a novel federated algorithm that solves the joint inverse optimization problem without data centralization.

\subsection{Problem statement}
This paper proposes a federated learning algorithm for solving a linear least squares problem over a convex feasible region $\mathcal W$ defined by nonnegativity constraints and a specific relationship among variable subvectors:
\vspace{-4mm}
\begin{subequations}
\label{basic_model}    
\begin{align}
\min_w \ & f(w) := \sum_{i=1}^N \|Aw_i - b_i\|^2  \\
\mbox{s.t.} \ & w \in \mathcal{W} := \Big\{ w := [w_1; \ldots ; w_N] \in \mathbb{R}^{Nn}_+:  w_N = \sum_{i=1}^{N-1} c_i w_i \Big\}
\end{align}
\end{subequations}
where \(N\) is the number of data sources (referred to as agents in the federated framework), and \(w_i \in \mathbb{R}^n_+\) is a non-negative decision variable vector corresponding to the reconstructed image for each agent \(i\). The matrix \(A\) and vector \(b_i\) are given, with \(A\) representing the discrete Radon transform and \(b_i\) denoting the measurements obtained by agent \(i\). The vector \(c := [c_1; \ldots; c_{N-1}] \in \mathbb{R}_+^{N-1}\) defines the relationship among the subvectors, satisfying \(\sum_{i=1}^{N-1} c_i \in (0, 1]\). The equality constraint ensures that \(w_N\) is a linear combination of the subvectors \(w_1, \ldots, w_{N-1}\). Further details on this formulation are provided in \S\ref{sec:method}.

\subsection{Related literature}

Distributed image reconstruction has been considered in a limited number of literature. For instance, alternating direction method of multipliers (ADMM) and its variants have been developed for distributed image reconstruction in astronomy \cite{ferrari2014distributed} and medicine \cite{wang2013distributed,wang2015distributed}. The existing work mainly aims to address large-scale data for image reconstruction by distributing algorithmic steps to multiple compute nodes \cite{wang2013distributed,ferrari2014distributed,majchrowicz2020multi}. Unlike the existing distributed algorithms, the federated algorithm in this paper considers distributed data sources without centralizing them. 
Moreover, the existing work considers uni-modal data, whereas our work requires modeling of the relationship among the different modalities by an equality constraint. 
Existing methods (e.g., \cite{di2016optimization,di2017joint,schwartz2024imaging}) for image reconstruction from multimodal data assume data collocation, not immediately applicable to the federated setting in this paper.

Most existing FL algorithms cannot be directly applicable to solving \eqref{basic_model}.
This limitation arises because these algorithms train a global model $w$ by leveraging local model parameters $\{w_i\}$ trained individually such that $w = w_1 = \ldots = w_N$. 
For example, the widely used federated averaging algorithm and its variants \cite{mcmahan2017communication, karimireddy2020scaffold,li2021fedbn,zhang2021fedpd,li2020federated} achieve this by averaging local model parameters at the server.
In contrast, our model aims to enhance local models by incorporating domain knowledge, specifically that the local models are interconnected through a linear combination, expressed by the constraint $w_N = \sum_{i=1}^{N-1} c_i w_i$.
Such \emph{global} constraints differ fundamentally from the local constraints typically imposed on individual agents in the FL literature \cite{di2023ppfl, ryu2022differentially, feng2020learning, lin2017collaborative}.
Examples of global constraints in FL include fairness considerations \cite{shen2021agnostic} and handling class imbalances \cite{shen2021agnostic}. FL algorithms designed to handle global constraints, such as projected gradient descent (PGD) \cite{chu2021fedfair} and the augmented Lagrangian method \cite{shen2021agnostic}, are not well-suited to our setting. Specifically, the projection step in PGD requires solving a convex optimization problem at each iteration, which can be computationally expensive and necessitates access to advanced optimization solvers—resources that may not be available on servers restricted to simple vector operations.
Similarly, the augmented Lagrangian approach introduces inefficiencies for our problem. This method results in a multi-block ADMM, where the number of blocks increases with the number of agents ($N$), necessitating serial updates for each block. Furthermore, as highlighted in \cite{chen2016direct}, ADMM may fail to converge when $N \geq 3$.
To address these challenges, we propose a novel FL algorithm tailored to solve \eqref{basic_model}. Our method requires only simple vector operations at the server while guaranteeing convergence.

\subsection{Contributions}
We summarize the contribution of our work as follows: 
\begin{itemize}[left=0em]
  \item {\emph{Multimodality modeling.}} We formulate a joint inverse optimization problem for reconstructing images from XRT and XRF tomographic datasets as a constrained least squares problem, leveraging their multimodal characteristics through linear combination constraints grounded in their physical relationships.
  \item {\emph{Federated algorithm.}} We propose a federated algorithm, termed \texttt{FIRM},  which solves the proposed \emph{constrained} least squares problem via local gradient computations and simple vector operations at the server. 
  \item {\emph{Convergence analyses.}} We establish the connection between \texttt{FIRM} and the well-established quadratic penalty (QP) methods. Built upon this connection, we propose an adaptive step size rule that guarantees the convergence to an optimal solution with a sublinear rate of $\mathcal{O}(1/\epsilon^2)$, and provides condition for which the rate becomes $\mathcal{O}(1/\epsilon)$, matching that of the projected gradient method. The connection also inspires an augmented Lagrangian extension, termed \texttt{FIRM}$^+$. 
  \item {\emph{Numerical experiments}} 
 demonstrate (i) a remarkable computational efficiency and stability of \texttt{FIRM}, and (ii) the enhanced quality of image reconstruction achieved through multimodality under various noise levels and algorithmic settings.
\end{itemize}


\subsection{Notation and organization of the paper}
We denote by \(\mathbb{Z}_{\ge 0}\) the set of nonnegative integers. For a positive integer $N$, we define \([N] := \{1, 2, \ldots, N\}\). We use $I$ to denote the identity matrix, $\boldsymbol{1}$ to represent a vector of all ones, and $\mbox{diag}(A)$ to indicate a block diagonal matrix with $A$ on its diagonal; their dimensions are either clear from the context or specified explicitly using subscripts.
We use \(\lambda_{\max}(A)\) and \(\lambda_{\min}(A)\) to indicate the largest and the smallest eigenvalues of matrix \(A\), respectively. 
For a vector \(v\), \(\|v\|\) represents the Euclidean norm of \(v\), i.e., \(\|v\| = \sqrt{v^T v}\), unless otherwise stated, and for a matrix \(A\), \(\|A\|\) denotes the spectral norm of \(A\), i.e., \(\|A\| = \sqrt{\lambda_{\max}(A^T A)}\). 
 For a vector $v \in \mathbb R^n$ and an index set $\mathcal J \subseteq [n]$, $v_\mathcal J$ denotes the subvector of $v$ corresponding to the index set $\mathcal J$. 
 For a convex set \(\mathcal{W}\) and a point \(v\), \(\mathcal{P}_{\mathcal{W}}(v)\) represents a projection operator that projects the point \(v\) onto the set \(\mathcal{W}\) in the Euclidean sense, i.e., \(\mathcal{P}_{\mathcal{W}}(v) = \arg\min_{w \in \mathcal{W}} \|w - v\|^2\); this projection is well-defined for convex $\mathcal W$, as the minimizer for the projection problem is known to exist and be unique. For cases where \(\mathcal{W}\) is the nonnegative orthant, we define \((\cdot)_+ := \mathcal{P}_{\mathcal{W}}(\cdot)\). We let $\mathcal {N}_{\ge 0}(u^*)$ denotes the normal cone of the nonnegative orthant at $u^*$, defined as $\{s: s^\top (u-u^*) \le 0, \forall u \ge 0\}$. For a scalar $\epsilon > 0$, $\mathcal{B}(\epsilon)$ represents the Euclidean norm ball of radius $\epsilon$ centered at origin, i.e., $\mathcal{B}(\epsilon) = \{u : \|u\|\le \epsilon\}$.

 For the rest of the paper, we present the mathematical model for federated multimodal image reconstruction in \S\ref{sec:method}, propose federated algorithms along with their convergence analyses in \S\ref{sec:alg}, and demonstrate the effectiveness of our approach through numerical experiments in \S\ref{sec:experiments}. Finally, we conclude the paper in \S\ref{sec:conclusions}.

\section{Mathematical Model}
\label{sec:method} We formulate the problem of reconstructing images from multimodal datasets, specifically the XRT and XRF datasets, within the structure of \eqref{basic_model}. We consider a simplified 2D model of an object. Specifically, each beamlet is characterized by a pair $(\theta, \tau)$, where $\theta$ denotes the X-ray beam angle and $\tau$ indicates the beamlet index. Let $\Theta$ and $\mathcal{T}$ represent the sets of beam angles and beamlet indices, respectively. To model the interaction of each beamlet with the object, the object is discretized into spatial voxels. Let $v$ index the set $\mathcal{V}$ of pixels used for this discretization (see Fig.~\ref{fig:xrt}). The intersection length (in $\mathrm{cm}$) of beamlet $(\theta, \tau) \in \Theta \times \mathcal T$ with pixel $v \in \mathcal V$ is characterized by $A \in \mathbb{R}^{|\Theta||\mathcal{T}| \times |\mathcal{V}|}$, representing the discrete Radon transform operator \cite{radon1986determination} for a drift-free experimental configuration.

\begin{figure}[h!]
\centering
\includegraphics[width=0.35\textwidth]{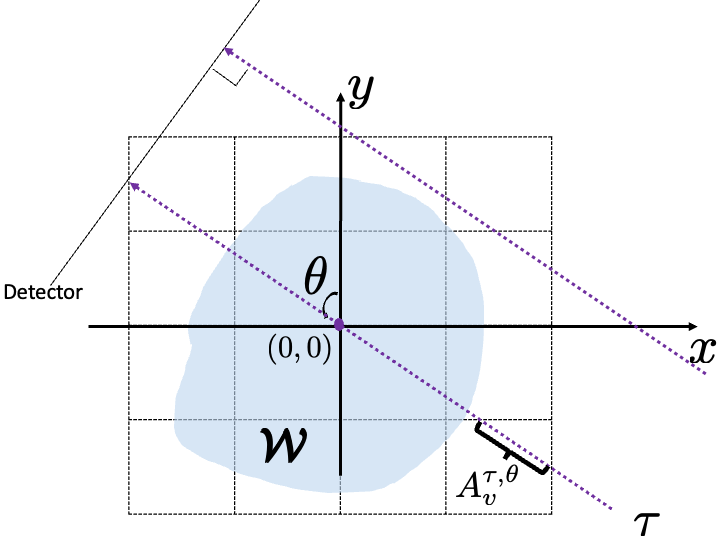}   
\caption{Illustration of the discrete tomographic geometry.}
\label{fig:xrt} 
\end{figure}

In principle, XRF is a process involving the emission of characteristic X-rays from a material, which can be analyzed to determine its elemental composition.
Let $\mathcal E$ denote the collection of the elements. 
We use $w_e \in \mathbb R^{|\mathcal V|}$ to denote its density (in $\mathrm{g}\,\mathrm{cm}^{-3}$) of element $e \in \mathcal E$ across the pixels $\mathcal V$, and let $w = (w_e)_{e \in \mathcal E}$. 
Following common practice, we assume XRF follows a linear forward process \cite{de2010quantitative} governed by the discrete Radon transform $A$. Therefore, for each $e \in \mathcal E$, given the flourescence data $b_{\texttt{XRF},e}\in \mathbb{R}^{|\Theta||\mathcal T|}$, the XRF model is expressed as $f_{\texttt{XRF}}(w_e)=Aw_e=b_{\texttt{XRF},e}$.

On the other hand, XRT is a process of X-ray attenuation that captures the amount of X-ray absorption by an object, which depends on the object’s density and composition. The goal of reconstruction from XRT data is to recover the linear attenuation coefficient $\mu_v$ (in cm$^{-1}$) for each pixel $v$ from the measurement $b_{\texttt{XRT}}\in \mathbb{R}^{|\Theta||\mathcal T|}$; we let $\mu = (\mu_v)_{v \in \mathcal V}$. The corresponding XRT model is $I_0\exp\left\{-A\mu\right\}=b_{\texttt{XRT}}$, where $I_0$ is the given incident beam energy, and the exponential is applied componentwise. By applying a logarithmic transformation, we convert the original nonlinear XRT model to its linear form as $f_{\texttt{XRT}}(\mu)=A\mu=-\log(\frac{1}{I_0}b_{\texttt{XRT}})$, where the logarithm is again applied componentwise.

Now, we establish a link between XRT and XRF modalities to fully exploit the multimodal benefit, allowing one modality to potentially compensate for the ill-posedness of the other. Specifically, the attenuation coefficient vector $\mu$ depend on the density vector $w$ as follows: $\mu=\sum_{e \in \mathcal E} c_e w_{e}$, where $c_e\in \mathbb{R}$ denotes the mass attenuation coefficient (in cm$^2$g$^{-1}$) of element $e \in \mathcal E$. The vector $c:=(c_e)_{e \in \mathcal E}$ is well characterized for individual elements and is tabulated in various sources \cite{thompsonx}. Using this relationship, the original XRT model can be reformulated in terms of the elemental density $w$; specifically, $f_{\texttt{XRT}}(w) = A(\sum_{e \in \mathcal E}c_e w_e) = -\log(\frac{1}{I_0}b_{\texttt{XRT}})$. Instead of solving individual modality as a common practice, we can then formulate the following joint inverse problem to reconstruct an image using both multimodal XRF and XRT datasets:
\begin{equation}
\label{eq:joint}
\min_{w\geq 0} \sum_{e \in \mathcal E}\|A w_e-b_{\texttt{XRF},e}\|^2 + \|A(\sum_{e \in \mathcal E} c_e w_e)+\log(\frac{1}{I_0}b_{\texttt{XRT}})\|^2,
\end{equation}
where the nonnegative constraint on $w$ is a consequence of natural physics on the property of interests.

In many applications, individual XRF data $b_{\texttt{XRF},e}$ for each element $e \in \mathcal E$ and XRT data $b_{\texttt{XRT}}$ are acquired at different times and using different instruments. We assume that each dataset is accessible only through a separate local agent. Let $N$ represent the total number of local agents, such that $N = |\mathcal{E}| + 1$. In this framework: (i) each local agent $i \in [N-1]$ has access to distinct $b_{\texttt{XRF},e}$ for some $e \in \mathcal{E}$; we let $b_i = b_{\texttt{XRF},e}$; (ii) the $N$-th local agent corresponds to the XRT data, with the dataset defined as $b_N = -\log\left(\frac{1}{I_0}b_{\texttt{XRT}}\right)$.  
This formulation allows the multimodal joint inverse problem to be modeled in the form of \eqref{basic_model}.

\section{Federated Algorithms}
\label{sec:alg}
Our goal is to solve the optimization problem \eqref{basic_model} within a federated setting, where each agent $i \in [N]$ has access only to its local data $b_i$. 
To accomplish this, we interpret the problem \eqref{basic_model} as a constrained federated learning (FL) model and propose tailored FL algorithms. 
Specifically, our FL algorithms are designed for a typical FL scenario where (i) each agent $i \in [N]$ can compute a feasible local solution $v_i \approx \argmin_{v} \|Av-b_i\|^2$ using its own data and computational resources, and (ii) a central server exists to compute an optimal solution $\{w_i^*\}_{i=1}^N$ of \eqref{basic_model} by performing simple vector operations on the local solutions $\{v_i\}_{i=1}^N$.

The projected gradient descent (PGD) algorithm, a common approach for solving constrained convex optimization problems, can be adapted to the federated setting, denoted as \texttt{FedPGD}. In this context, each agent $i \in [N]$ performs a local gradient step, while the projection step is handled by a central server after collecting local solutions $v_i$ from all agents. Specifically, the following two steps are repeated, where $t\in \mathbb Z_{\ge 0}$ denotes the iteration round:
\begin{subequations}
\label{pgd}    
\begin{align}
& v_i^t \gets w_i^t - \alpha_t g_i^t, \ \forall i \in [N], \label{pgd_1}   \\
& w^{t+1}  \gets \mathcal{P}_{\mathcal{W}} (v^t) = \argmin_{w \in \mathcal{W}} \|w - v^t\|^2. \label{pgd_2}
\end{align}
\end{subequations}
Here, 
$w^t_i$ is the current solution of an agent $i \in [N]$, 
$\alpha_t > 0$ is a step size,
$g_i^t := g_i(w_i^t)=2 A^T (Aw_i^t-b_i)$ is the gradient of $\|Aw_i-b_i\|^2$ at $w_i^t$, and
$v^t := [v^t_1; \ldots; v^t_N]$ is the concatenated column vector. 

However, \texttt{FedPGD} may not be efficient for solving \eqref{basic_model} due to the expensive projection step \eqref{pgd_2} at the server, which requires more than simple vector operations and can become computationally burdensome to conduct at each iteration (see, e.g., Remark \ref{rema:f-PGD}). To address this issue, we propose an algorithm that \emph{replaces the costly projection step with simpler vector operations}. 

For the remainder of this section, we present an FL algorithm for solving \eqref{basic_model} in \S\ref{sec:prop_algo}, along with its convergence analysis by establishing a connection to the quadratic penalty (QP) method in \S\ref{sec:relation_penalty}. Furthermore, we outline an extension of the algorithm to the augmented Lagrangian (AL) method in \S\ref{sec:augmented_lagrangian}.

\subsection{Proposed Algorithm} \label{sec:prop_algo}

We propose an algorithm, denoted as \texttt{FIRM}, for solving \eqref{basic_model}, that replaces the expensive projection step \eqref{pgd_2} in \texttt{FedPGD} with simple vector operations as follows: 
\begin{subequations}
\label{our_approach}
\begin{align}
&x^t \gets \frac{1}{2}  ( y^t + v^t_N ) \text{ where } y^t := \sum_{i=1}^{N-1} c_i v^t_i\label{our_approach_avg} \\ 
& z^t_i \gets
\begin{cases}
v^t_i + c_i (x^t - y^t) & \text{ for } i \in [N-1], \\
x^t & \text{ for } i = N,
\end{cases}  \label{our_approach_update} \\
& w^{t+1}_i \gets (z^t_i)_+, \ \forall i \in [N], \label{clipping}
\end{align}
\end{subequations}
where $(\cdot)_+$ denotes the projection onto the nonnegative orthant.

This is inspired by the widely-used federated averaging algorithm, which updates the global model at the server by averaging local models in each round $t \in \mathbb{Z}_{\geq 0}$.
Specifically, after receiving the local solution $v^t_i$ following step \eqref{pgd_1}, the server computes an average of the last agent's local solution $v^t_N$ and a linear combination of the remaining $N-1$ local solutions, as given by \eqref{our_approach_avg}. The server then updates the local solutions of the first $N-1$ agents based on the current disagreement, measured by $x^t - y^t = \frac{1}{2}(v_N^t - y^t)$, scaled by each agent's contribution, approximated by $c_i$. The solution of the last agent is replaced with the averaged solution $x^t$, as described in \eqref{our_approach_update}.
Next, each element of the updated local solution, $z_i^t$, is clipped to satisfy the nonnegativity requirement on $w$, as shown in \eqref{clipping}. The resulting $w_i^{t+1}$ is then sent to each agent $i\in[N]$ for use in local training in the following round $t+1$. An overview of our proposed \texttt{FIRM} is provided in Figure \ref{fig:overview_firm}.

\begin{figure}[h!]
\centering
\includegraphics[width=0.5\textwidth]{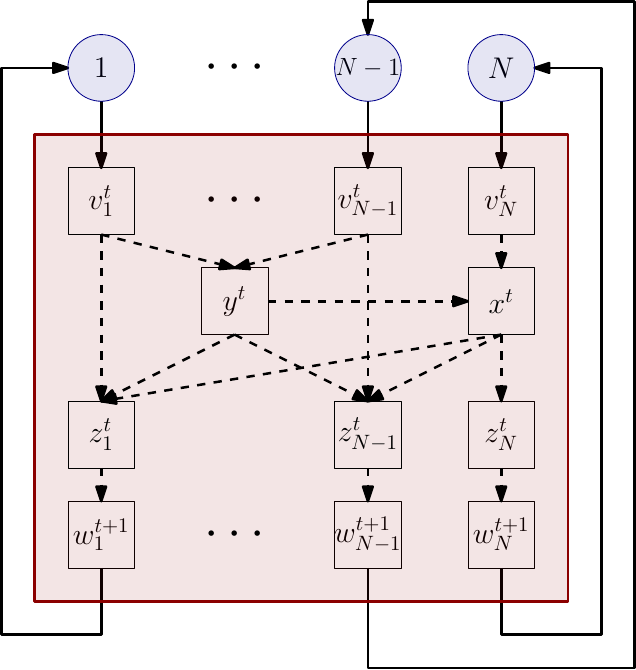}   
\caption{The overview of \texttt{FIRM}. At each round $t$, $N$ agents (represented by blue circles) send their local solutions $\{v^t_i\}_{i=1}^N$ computed by the formulation \eqref{pgd_1} to the server (represented by a red box). Then the server conducts a set of vector operations as in \eqref{our_approach}.
The resulting $w^{t+1}_i$ is then sent to the agent $i \in [N]$ which will be used as an initial point for local training in the next round.}
\label{fig:overview_firm} 
\end{figure}

\begin{remark}[Less computation load at the server compared to \texttt{FedPGD}]
The vector operations and clipping steps outlined in \eqref{our_approach} can be significantly more computationally efficient than the projection step described in \eqref{pgd_2}. The projection step in \eqref{pgd_2} involves projecting a vector onto a polyhedron $\mathcal W$, which has a complexity of $O(N^3n^3)$. In contrast, the computations required for \eqref{our_approach} scale with $O(Nn)$. Moreover, the proposed algorithm can fully utilize GPU capabilities, as the steps in \eqref{our_approach} can be performed in parallel using multiple GPU threads. This is in contrast to \texttt{FedPGD}, where the projection step \eqref{pgd_2} may require optimization solvers that run on CPUs. Consequently, this may incur significant overhead due to the time required for transferring information between the CPU and GPU, such as (i) $v^t$ in \eqref{pgd_1} from GPU to CPU, and (ii) $w^{t+1}$ in \eqref{pgd_2} from CPU to GPU.
\label{rema:f-PGD}
\end{remark}

By carefully designing the step size update, i.e., $\alpha_t$ in \eqref{pgd_1}, we ensure that the \emph{proposed procedure converges to an optimal solution of \eqref{basic_model}, without negatively impacting the convergence rate of \texttt{FedPGD} under a certain condition}, while significantly reducing the computational burden at each iteration. The complete algorithm, including the step size rule, is presented in Algorithm~\ref{algo}. As will be discussed later in \S\ref{sec:relation_penalty}, \texttt{FIRM} can be interpreted as an inexact QP method, where a sequence of penalized versions of a modified \eqref{basic_model} is solved with increasingly large penalties \(1/\eta_k\) imposed on constraint violations. In this framework, the outer loop updates the penalty parameter \(\eta_k\) based on a given decay rate \(r\). For a fixed \(\eta_k\), the inner loop (with iteration counter \(t\)) solves the corresponding penalized problem to a specified accuracy. In \texttt{FIRM}, the penalty parameter \(\eta_k\) governs both the termination tolerance of the inner loop (Line~\ref{line:inner-termination}) and the step size (Line~\ref{line:agent}).
\begin{algorithm}
\caption{\texttt{FIRM}} 
\begin{algorithmic}[1]
\Require Iteration counter for the outer-loop $k \gets 1$; initial penalty parameter $\eta_1 \gets \beta \frac{2-\sum_{i=1}^{N-1}c_i^2}{4\lambda_{\max}(A^\top A)}$ for some constant $\beta \geq 1$; initial solution $u^0 \in \mathbb R^{Nn}$; penalty parameter decay rate $r \in (0,1)$; termination tolerance $\epsilon >0$
\While{$\eta_k \ge \epsilon$}
\State $t \gets 0$; $w^{t}\gets u^{k-1}$ 
\While{\texttt True}
\State Send $w^t_i$ and $\alpha_{t} = \eta_k$ to each agent $i \in [N]$; (each agent performs \eqref{pgd_1})\label{line:agent}
\State Collect $v^t$ from agents and perform \eqref{our_approach} \label{line:server}
\If{$\|w^{t+1} - w^t\| \le \eta_k^2$}\label{line:inner-termination}
\State \texttt{Break}
\EndIf
\State $t \gets t+1$
\EndWhile
\State $u^{k} \gets w^{t+1}$
\State $\eta_{k+1} \gets r \eta_k$ \label{line:update}
\State $k \gets k+1$; 
\EndWhile
\end{algorithmic}
\label{algo}
\end{algorithm}


\subsection{Relation of \texttt{FIRM} to the Quadratic Penalty Method} \label{sec:relation_penalty}
In this section we demonstrate that \texttt{FIRM} falls within the class of quadratic penalty (QP) methods. Building on this connection, we establish the convergence of \texttt{FIRM} to an optimal solution of \eqref{basic_model} and its rate of convergence. For clarity in presentation, we introduce an $n \times Nn$-matrix $D = [-c_1I_n ~ -c_2 I_n ~ \cdots ~ -c_{N-1}I_n ~ I_n]$,
which allows us to express
\begin{align}
D w = w_N - \sum_{i=1}^{N-1} c_i w_i, \label{def_Delta}
\end{align}
representing the residual of the linear combination constraint applied to the vectors $\{w_i\}_{i=1}^N$.
Note that \eqref{our_approach} can be rewritten as 
\begin{align*}
w^{t+1}_i &\gets (z_i^t)_+  \gets \begin{cases}
(v_i^t + c_i(x^t-y^t))_+\\
(x^t)_+
\end{cases} = \begin{cases}
(v_i^t + \frac{c_i}{2}Dv^t)_+\\
(v_N^t - \frac{1}{2}Dv^t)_+.
\end{cases}
\end{align*}
Therefore, we rewrite Lines \ref{line:agent}-\ref{line:server} in Algorithm \ref{algo} as follows:
\begin{subequations}
 \label{firm}   
\begin{align}  
& w^{t+1}_i \gets \Big( w^t_i - \eta_k g^t_i + d^t_i \Big)_+,
\end{align}
where 
\begin{align}
d^t_i =
\begin{cases}
\frac{c_i}{2} Dv^t & \text{ for } i \in [N-1],\\ 
- \frac{1}{2} Dv^t & \text{ for } i = N.
\end{cases}      
\end{align}
\end{subequations}

\begin{proposition}
\texttt{FIRM} is an inexact first-order QP method for solving the following optimization model:
\begin{subequations}
\label{modified_model}
\begin{align}
\min_{w \in \mathbb{R}^{Nn}_+} \ & F(w):= \sum_{i=1}^{N} \|A w_i - b_i\|^2 - \frac{1}{2} \|A Dw - Db\|^2 \\
\mbox{s.t.} \ & Dw = 0,\label{penalized-constr}
\end{align}    
\end{subequations}
which is equivalent to the original model \eqref{basic_model}.
\label{prop:qp}
\end{proposition}
\begin{proof}

The equivalence of models \eqref{modified_model} and \eqref{basic_model} is clear, as they share the same feasible region, and for any feasible $w$, the additional objective term, $-\frac{1}{2}\|Db\|^2$, remains constant.
Consider a quadratic penalty function of \eqref{modified_model}, where the penalty is imposed for the violation of \eqref{penalized-constr} with a penalty parameter $\frac{1}{4\eta}$:
\begin{align}
    Q(w;\eta)  := F(w)  + \frac{1}{4\eta} \|Dw\|^2. \label{quadratic_penalty_function}
\end{align}
QP methods solve \eqref{modified_model} by solving a series of penalized problems 
\begin{equation*}\min_{w \in \mathbb{R}^{Nn}_+} Q(w;\eta_k),\tag{\mbox{$P_{\eta_k}$}}\end{equation*}
where the penalty weight $\frac{1}{4\eta_k}$ increases to infinity, i.e., $\eta_k \rightarrow 0$. Consider an inexact first-order QP method that solves each penalized problem $(P_{\eta_k})$ approximately via several PGD steps with a step size of $\eta_k$. Specifically, for each $k$, a subroutine of PGD steps is performed, with the iteration counter $t \in \mathbb Z_{\ge 0}$, as follows: 
{\footnotesize\begin{align}
w^{t+1} & \gets \left( w^t - \eta_k \left\{ \begin{bmatrix}2A^T(Aw^t_1 - b_1) \\ \vdots \\ 2A^T(Aw^t_N - b_N)\end{bmatrix} - D^\top A^{\top}(A D w^t - Db)+ \frac{1}{2\eta_k} D^\top D w^t\right\}   \right)_+\\
&= \left( w^t - \eta_k g^t + \eta_k D^\top A^{\top}(A Dw^t - Db)  - \frac{1}{2} D^\top D w^t \right)_+,
\label{pgd_quadratic_penalty}
\end{align}}
where $g^t := [g_1^t; \cdots; g_N^t]$ is a concatenated column vector and $g_i^t$ is as defined in \eqref{pgd_1}.

We show that \eqref{pgd_quadratic_penalty} is equivalent to \eqref{firm}.
To see this, we rewrite \eqref{firm} as follows:
\begin{subequations}
\label{firm_derivation}
\begin{align}
w^{t+1} \gets \ &\Big( w^t - \eta_k g^t -\frac{1}{2} D^\top D v^t \Big)_+  \\
= \ & \Big( w^t - \eta_k g^t - \frac{1}{2} D^\top (D w^t - \eta_k D g^t) \Big)_+  \mbox{ since } Dv^t = Dw^t - \eta_k Dg^t,\label{firm:equiv}
\end{align} 
\end{subequations}
Note that 
{\footnotesize\begin{align*}
Dg^t = 2D
\begin{bmatrix}A^\top A \\
& \ddots\\
& & A^\top A\end{bmatrix}
w^t - 2D \begin{bmatrix}A^\top b_1\\\vdots \\ A^\top b_N \end{bmatrix} = 2 A^\top \left( A Dw^t - 2 Db \right).
\end{align*}}
Therefore, \eqref{firm:equiv} is equivalent to \eqref{pgd_quadratic_penalty}.
\end{proof}

\subsection{The choice of initial step size $\eta_1$}
The proof of Proposition \ref{prop:qp} demonstrates that \texttt{FIRM} is an inexact first-order QP method, where each penalized problem $(P_{\eta_k})$ is approximately solved after several PGD steps with a specified step size of $\eta_k$. For the PGD step to ensure descent for $(P_{\eta_k})$, the step size $\eta_k$ must satisfy certain conditions. We first present a variant of the descent lemma for $L$-smooth functions, with the proof provided in Appendix \ref{appendix:proof-descent-lemm}.
\begin{lemma} Consider a problem $\min_{w \ge 0} Q(w)$, where $Q$ is an $L$-smooth function. If the step size $\eta \le \frac{3}{2L}$, then each projected gradient step yields descent in $Q$, provided $w$ is not optimal. Specifically, for $w'=(w- \eta\nabla Q(w) )_+$, we have:
$$Q(w') \le Q(w) - \frac{1}{4\eta} \|w' - w\|^2.$$\label{lemm:descent}
\end{lemma}

By leveraging Lemma \ref{lemm:descent}, we can show that starting with the initial stepsize $\eta_1 = \beta \frac{2-\sum_{i=1}^{N-1}c_i^2}{4\lambda_{\max}(A^\top A)}$ 
and following the update rule in Line \ref{line:update}, each iteration of \texttt{FIRM} eventually produces a descent for the current penalized problem $(P_{\eta_k})$:
\begin{proposition} 
    For any $\eta_k  \in \left(0,\frac{2-\sum_{i=1}^{N-1}c_i^2}{4\lambda_{\max}(A^\top A)}\right]$, \eqref{firm} guarantees a descent for $(P_{\eta_k})$.
    \label{prop:descent-stepsize}
\end{proposition}
\begin{proof}
Note that
{\footnotesize
\begin{align}
& \nabla^2 Q_{\eta_k} := \nabla^2 Q(w; \eta_k) = \underbrace{
\bar A^{\top} (2I_m - {\bar D}^\top \bar  D) \bar A}_{=:B} + \frac{1}{2\eta_k} D^\top D, \label{eq:hessian} 
\end{align}}
where
\(\bar A:= 
\mbox{diag}(A)
,\ \bar D = [-c_1I_m ~ -c_2 I_m ~ \cdots ~ -c_{N-1}I_m ~ I_m].\) 
Note that $B \succeq 0$ when $\lambda_{\min}(2I_m-{\bar D}^\top {\bar D})=2-\lambda_{\max}(\bar{D}^\top \bar{D}) \ge 0$. We establish that $\lambda_{\max}(\bar{D}^\top \bar{D}) = \lambda_{\max}({D}^\top D) = \sum_{i=1}^{N-1}c_i^2 + 1$ from the following:
{\footnotesize$$(\lambda_{\max}(D^\top D))^2 = \lambda_{\max}(D^\top DD^\top D) = (\sum_{i=1}^{N-1}c_i^2 + 1) \lambda_{\max}(D^\top D),$$}
where the first equality follows from the property that $\lambda_{\max}(C^2) = (\lambda_{\max}(C))^2$ for a symmetric $C \succeq 0$ and the second equality from $DD^\top = (\sum_{i=1}^{N-1}c_i^2 + 1) I_n$. Since $\lambda_{\max}(D^\top D) > 0$, it follows that $\lambda_{\max}(D^\top D) = \sum_{i=1}^{N-1}c_i^2 + 1$. The same argument applies to $\lambda_{\max}({\bar D}^\top \bar D)$. Therefore, $2-\lambda_{\max}({\bar D}^\top {\bar D}) = 1-\sum_{i=1}^{N-1}c_i^2 \ge 0$. 
Consequently, $\nabla^2 Q_{\eta_k} \succeq 0$, and thus $Q(w, \eta_k)$ is convex quadratic. 

For a convex quadratic function, it is $L$-smooth for any $L$ that is greater than or equal to the largest eigenvalue of its Hessian. Note that the following holds:
{\footnotesize\begin{align*}\lambda_{\max}(\nabla^2 Q_{\eta_k}) & \le \lambda_{\max}(B) + \lambda_{\max}(\frac{1}{2\eta_k}D^\top D) \le 2\lambda_{\max}(A^\top A) + \frac{1}{2\eta_k}(\sum_{i=1}^{N-1}c_i^2 + 1).
\end{align*}}
Therefore $Q(w;\eta_k)$ is $L(\eta_k)$-smooth for $L(\eta_k) = 2 \lambda_{\max}(A^\top A) + \frac{\sum_{i=1}^{N-1}c_i^2 + 1}{2\eta_k}$.


By Lemma \ref{lemm:descent}, if $\eta_k \in (0, \frac{3}{2L(\eta_k)}]$, then we have a decrease for each step, since
{\footnotesize\begin{align*}
\eta_k L(\eta_k) & \le  2\eta_k \lambda_{\max}(A^\top A) + \frac{1}{2}(\sum_{i=1}^{N-1}c_i^2 + 1) \le \frac{2-\sum_{i=1}^{N-1} c_i^2}{2}  + \frac{1}{2}(\sum_{i=1}^{N-1}c_i^2 + 1)  =  \frac{3}{2}.
\end{align*}}
Hence, the desired result follows.
\end{proof}

Therefore, by starting the algorithm with the specific $\eta_1$ and ensuring that the stepsize only decreases, \texttt{FIRM} is guaranteed to reduce the objective function value of the current penalized problem eventually. In the next section, we establish the connection between the approximate solutions to the penalized problems and the original problem, i.e., \eqref{modified_model}, prove its convergence to the optimal solution of \eqref{modified_model}, and establish its rate of convergence.

\subsection{Convergence results}
We first present the well-known optimality conditions for the penalized problem $(P_{\eta_k})$ and \eqref{modified_model}, respectively: 
\begin{multline} w^* \text{ is an optimal solution to } (P_{\eta_k}) \text{ if and only if }\\ w^* \ge 0 \text{ and }   \begin{cases}\nabla Q(w^*; \eta_k)_j =0, & \forall j: w^*_j > 0,\\ 
\nabla Q(w^*; \eta_k)_j \geq 0, & \forall j: w^*_j = 0. \end{cases} \end{multline}
\begin{multline} u^* \text{ is optimal to \eqref{modified_model} if and only if } \\
\exists \mu^*: u^* \ge 0, \quad \|Du^*\| = 0, \quad \nabla F(u^*) + D^\top \mu^* \in -\mathcal{N}_{\ge 0}(u^*).\label{eq:KKT}\end{multline}

\begin{remark}
    Under the weak Slater's condition, if the primal problem \eqref{basic_model} is feasible and bounded from below, an optimal dual variable $\mu^*$ is guaranteed to exist. This condition holds in our case.
    \label{rema:opt-lagrange-existence}
\end{remark}


This leads to the following definitions of sub-optimality for the penalized problem and for \eqref{modified_model}, respectively: \begin{definition} 
    A point $w'$ is an $\eta_k$-optimal solution to $(P_{\eta_k})$ if it satisfies: \begin{equation} w' \geq 0, \quad \|\nabla Q(w'; \eta_k)_{\mathcal J(w')}\| \leq \eta_k, \quad \nabla Q(w'; \eta_k) \geq -\eta_k \mathbf{1},\label{term:opt} \end{equation} where $\mathcal J(w') := \{j \in [Nn] : w'_j > \eta_k^2\}$.
    \label{def:sub-opt-penalty}
\end{definition}
\begin{definition}
    A point $u'$ is an $(\epsilon_p, \epsilon_s)$-optimal solution to \eqref{modified_model}, if it satisfies:
    $$\exists \mu' : u' \ge 0, \quad \|Du'\| \le \epsilon_p, \quad \nabla F(u') + D^\top \mu' \in -\mathcal N_{\ge 0}(u')+\mathcal B(\epsilon_s).$$
\end{definition}
We now demonstrate that, for each $k \in \mathbb Z_{\ge 0}$, the termination of the inner loop in Algorithm \ref{algo} indicates an identification of an $\eta_k$-optimal solution to $(P_{\eta_k})$:
\begin{lemma}
Let $w_+ = (w - \eta_k\nabla Q(w; \eta_k))_+$. If $\|w - w_+\| \le \eta_k^2$, then $w$ is an $\eta_k$-optimal solution to $(P_{\eta_k})$.
\label{lemm:term-opt}
\end{lemma}
\begin{proof}
Let $\mathcal J:= \mathcal J(w)$ and $\nabla Q := \nabla Q(w; \eta_k)$. From $\sum_{j \in \mathcal J} (w_j - w_{+j})^2 + \sum_{j \not\in \mathcal J} (w_j - w_{+j})^2 \le \eta_k^4$, we have
\begin{subequations}
\begin{align}
\forall j \not\in \mathcal J, |w_j - w_{+j}| \le \eta_k^2 ,\label{opt-sol:1}\\
\sum_{j \in \mathcal J} (w_j - w_{+j})^2 \le \eta_k^4.\label{opt-sol:2} 
\end{align}    
\end{subequations}
Note that, from \eqref{opt-sol:1}, for each $j \not\in\mathcal J$, if $w_{+j} > 0$, then $|w_j-w_{+j}| = |\eta_k\nabla Q_j| \le \eta_k^2$, and thus $\nabla Q_j \ge -\eta_k$. On the other hand, if $w_{+j} = 0$, it follows that $\eta_k \nabla Q_j \ge w_j \ge 0$, and thus $ \nabla Q_j \ge 0 > -\eta_k$. In addition, from \eqref{opt-sol:2}, $w_{+j}$ must be nonzero for all $j \in \mathcal J$, since otherwise, $(w_j)^2 \le \eta_k^4$, which is a contradiction. Thus, we have $\|\eta_k \nabla Q_{\mathcal J}\|^2 \le \eta_k^4$, which implies $\|\nabla Q_{\mathcal J}\| \le \eta_k$. 
\end{proof}

It is well-known that the gradient mapping with a stepsize of $\frac{1}{L}$ for $L$-smooth convex functions exhibits a monotonicity property. We establish that this property also holds with a different stepsize for convex quadratic functions, the proof of which is provided in Appendix \ref{proof:monotone}:
\begin{lemma}
Let $Q$ be an $L$-smooth convex quadratic function, and $w_+ = (w - \eta\nabla Q(w))_+$ for some $\eta \in (0, \frac{2}{L}]$. Then, $\|w_+ - w_{++}\| \le \| w- w_+\|$, where $w_{++} = (w_+ - \eta \nabla Q(w_+))_+$.\label{lemm:monotone}
\end{lemma}
Therefore, the following corollary immediately follows from Lemmas \ref{lemm:term-opt} and \ref{lemm:monotone}:
\begin{corollary}
For each $k \in \mathbb Z_{\ge 0}$ in Algorithm \ref{algo}, $u^k$ is an $\eta_k$-optimal solution to $(P_{\eta_k})$.\label{coro:eta-opt}
\end{corollary}

Now, we establish the connection between an $\eta_k$-optimal solution to $(P_{\eta_k})$ and an optimal solution of \eqref{modified_model}: 
\begin{proposition}
 Let $u^{*}$ be any optimal solution to \eqref{modified_model}. For each $k \in \mathbb Z_{\ge 0}$, $u^k$ satisfies the following:
 \begin{equation}Q(u^k; \eta_k) - F(u^*) \le \eta_kC_k,
 \label{eq:opt-bound:1}
 \end{equation}
 where $C_k:=\frac{5}{2}\sqrt{Nn}\|u^k-u^*\| + \sqrt{Nn}\|\nabla F(u^*)\|\eta_k+ Nn\eta_k^2$.
\label{prop:opt-gap}
\end{proposition}
\begin{proof}
Let $F^*$ denote the optimal objective value of \eqref{modified_model}, i.e., $F^* = F(u^*)$, and let $\mathcal J_k$ denote $\mathcal J(u^k)$ and $\nabla Q^k := \nabla Q(u^k; \eta_k)$.  Note that from $Q(u^*;\eta_k) = F(u^*) +\frac{1}{4\eta_k}\|Du^*\|^2= F(u^*)$, we have the following:
{\footnotesize\begin{subequations}
\begin{align}
 Q(u^k; \eta_k) - F(u^*) = \ & Q(u^k; \eta_k) - Q(u^*; \eta_k)\\
\le \ & (\nabla Q^k)^\top (u^k - u^{*})\label{eq:opt-gap:0}\\
= \ & \sum_{j \in \mathcal J_k}\nabla Q^k_j (u^k_{j} - u_j^{*}) + \sum_{j \not\in \mathcal J_k}\nabla Q^k_j (u^k_{j} - u_j^{*})\\
\le \ & \eta_k \|(u^k - u^{*})_{\mathcal J_k}\| + \sum_{j \not\in\mathcal J_k}\max \big \{\nabla Q^k_j \eta_k^2, \ \eta_k u_j^{*} \big\} \label{eq:opt-gap:1}\\
\le \ & \eta_k \Big(\|(u^k - u^{*})_{\mathcal J_k}\| + \sum_{j \not\in\mathcal J_k}|\nabla Q^k_{j} \eta_k| +  u_j^{*}\Big)\label{opt-gap:2}\\
\le \ &  \eta_k \Big(\|(u^k - u^{*})_{\mathcal J_k}\|_1 + \eta_k\|\nabla Q^k\|_1 +  \|(u^{*} - u^k)_{\setminus\mathcal J_k}\|_1 + \|u^{k}_{\setminus\mathcal J_k}\|_1 \Big) \label{opt-gap:4}\\
\le \ &  \eta_k \Big(\|u^k - u^{*}\|_1 + \eta_k\|\nabla Q^k\|_1 + Nn\eta_k^2 \Big) \label{opt-gap:5}\\
\le \ &  \eta_k \Big(\frac{5}{2}\sqrt{Nn}\|u^k-u^*\| + \sqrt{Nn}\|\nabla F(u^*)\|\eta_k + Nn\eta_k^2\Big), \label{opt-gap:6}
\end{align}
\end{subequations}}
where \eqref{eq:opt-gap:0} follows from the convexity of $Q(\cdot; \eta_k)$ and \eqref{eq:opt-gap:1} is obtained from the Cauchy-Schwarz inequality, the fact that $u^k$ is an $\eta_k$-optimal solution to ($P_{\eta_k}$) from Corollary \ref{coro:eta-opt}, and by considering two cases for $j \not\in \mathcal J_k$: $\nabla Q^k_j < 0$ and $\nabla Q^k_j \ge 0$. We obtain \eqref{opt-gap:2} by factoring out $\eta_k$ and using $\max\{a, b\} \le |a| + |b|$. In \eqref{opt-gap:4}, the index set $\setminus \mathcal J_k$ denotes $[nN] \setminus \mathcal J_k$; the inequality follows from the well-known property $\|\cdot\| \le \|\cdot\|_1$ and the triangle inequality $\|a\|_1 = \|a-b+b\|_1 \le \|a-b\|_1 + \|b\|_1$. The inequality in \eqref{opt-gap:5} is obtained using the fact that $\|a\|_1 + \|b\|_1 = \|[a;b]\|_1$ along with the condition $0 \le u_j^k \le \eta_k^2$ for all $j \notin \mathcal J_k$. Finally, \eqref{opt-gap:6} holds since 
$\eta_k\|\nabla Q^k\| - \eta_k\|\nabla F(u^*)\|=\eta_k\|\nabla Q^k\| - \eta_k\|\nabla Q(u^*;\eta_k)\| \le \eta_k\|\nabla Q^k - \nabla Q(u^*;\eta_k)\| \le \eta_k L(\eta_k) \|u^k-u^*\| \le \frac{3}{2}\|u^k - u^*\|$ from the $L(\eta_k)$-smoothness of $Q(\cdot; \eta_k)$, and $\eta_k \|\nabla Q^k\|_1 \le \sqrt{Nn}\eta_k \|\nabla Q^k\|$ from the relationship between $l_1$- and $l_2$-norms. 
\end{proof}

From Proposition \ref{prop:opt-gap} and Propositions 3 and 10 in \cite{lan2013iteration}, we have the following:
\begin{proposition}Let $(u^{*}, \mu^*)$ be any optimal primal-dual solution pair to \eqref{modified_model}, i.e., it satisfies \eqref{eq:KKT}. For each $k \in \mathbb Z_{\ge 0}$, $u^k$ is a $\big(2\eta_k (2 \|\mu^*\| + \sqrt{C_k}), \frac{5}{2}\eta_k \big)$-optimal solution to \eqref{modified_model}, where $C_k$ is defined in Proposition \ref{prop:opt-gap}.
\label{prop:opt-gap-orig}
\end{proposition}
\begin{proof}
From Proposition \ref{prop:opt-gap} and following the same procedure in the proof of Proposition 10 in \cite{lan2013iteration}, we have
{\footnotesize\begin{equation*}\eta_kC_k \ge Q(u^k; \eta_k) - F(u^*) = F(u^k) + \frac{1}{4\eta_k}\|Du^k\|^2 - F(u^*) \ge -\|\mu^*\|\|Du^k\|+ \frac{1}{4\eta_k}\|Du^k\|^2,\end{equation*}}
where the last inequality is from Corollary 2 in \cite{lan2013iteration}. Then, by applying the quadratic formula to the one-dimensional equation in $\|Du^k\|$ and using the fact that $\|Du^k\| \ge 0$ and $\sqrt{a+b} \le \sqrt{a} + \sqrt{b}$ for nonnegative $a$ and $b$, we get
\begin{equation}
\|Du^k\| \le 2\eta_k \Big(2 \|\mu^*\| + \sqrt{C_k}\Big).
\label{eq:2}
\end{equation}
Now, let $\mu^k := \frac{1}{2\eta_k}D u^k$ and define $u^{k-}$ as the inner loop iterate directly preceding  $u^k$, i.e., $u^k=(u^{k-}- \eta_k \nabla Q(u^{k-}; \eta_k))_+$. Then, by applying Proposition 3(b) from \cite{lan2013iteration}:
{\footnotesize\begin{align}
\| \nabla \phi (\tilde{x})]^{\tau}_X \| \leq \epsilon \ \ \textit{ implies that } \ \ \nabla \phi (\tilde{x}^+) \in - \mathcal{N}_X(\tilde{x}^+) + \mathcal{B} \big( (1+\tau L_{\phi}) \epsilon \big) \nonumber
\end{align}}
with $\tau \gets \eta_k$, $L_\phi \gets L(\eta_k)$, $\epsilon \gets \eta_k$, $X \gets \mathbb R^{Nn}_+$, $\tilde x \gets u^{k-}$, $\|\nabla \phi(\tilde x)]^\tau_X \| \gets \frac{1}{\eta_k}\|u^k - u^{k-}\|$, $\nabla \phi(\tilde x^+) \gets\nabla Q(u^k;\eta_k) = \nabla F(u^k) + D^\top \mu^k$, we have 
\begin{equation}\nabla F(u^k) +  D^\top \mu^k \in -\mathcal{N}_{\ge 0}(u^k) + \mathcal B(\frac{5}{2}\eta_k).\label{eq:KKT:approx}\end{equation}
\end{proof}

Proposition~\ref{prop:opt-gap-orig} is essential to establish that every limit point of the iterates generated by \texttt{FIRM}, that is Algorithm~\ref{algo}, is an optimal solution of \eqref{basic_model}. In the following theorem, we formalize this result and demonstrate that, when the iterates remain bounded, \texttt{FIRM} converges to an $\epsilon$-optimal solution of ($P_\epsilon$) within $O(1/\epsilon^2)$ number of iterations. Let $\{u^{k}\}$ denote the sequence of iterates generated by the outer loop of Algorithm \ref{algo}. We further show that, under certain conditions, \texttt{FIRM} achieves $|Q(u^k;\epsilon) - F(u^*)| \le \epsilon$ in $O(1/\epsilon)$ iterations. This aligns with the iteration complexity of \texttt{FedPGD}, which achieves $F(u^k) - F(u^*) \le \epsilon$ after $O(1/\epsilon)$ iterations. We will discuss this further following the next theorem:
\begin{theorem}
Every limit point of $\{u^{k}\}$ is optimal to \eqref{modified_model}. Moreover, if $\{u^{k}\}$ is bounded, i.e., $\{u^k\} \subseteq \mathcal B(M)$ for some $M > 0$, then for any $\epsilon >0$, the sequence converges to an $\epsilon$-optimal solution of ($P_\epsilon$) after $O(\frac{1}{\epsilon^2})$ executions of \eqref{firm}.
\label{theo:convergence}
\end{theorem}
\begin{proof}

We show that any limit point of $\{u^k\}$ satisfies \eqref{eq:KKT}. Consider a limit point $\bar u$ of $\{u^{k}\}$, meaning that there exists a convergent subsequence $\{u^{k}\}_{\mathcal K} \subseteq \{u^{k}\}$ with $\lim_{k \in \mathcal K}u^{k}=\bar u$. From \eqref{eq:2}, we have:
\begin{subequations}
    \begin{align}
 \limsup_{k\in \mathcal K }\|D u^k\| \le 2\Big(2\|\mu^*\| + \sqrt{C^*}\Big)\lim_{k \in \mathcal K }\eta_k = 0,  \label{eq:limsup}
\end{align}
\end{subequations}
where $C^* := \limsup_{k \in \mathcal K}C_k = \frac{5}{2}\sqrt{Nn}\|\bar u-u^*\|$, and $\limsup \eta_k$ is replaced with $\lim \eta_k$ since $\eta_k$ is monotonic decreasing. Therefore, $\|D\bar u\| = 0$, meaning that the penalized constraint is met for any limit point. In addition, from \eqref{eq:KKT:approx} and the upper semicontinuity of normal cones of a convex set, we get
 $\nabla F(\bar u) +  D^\top \bar \mu \in -\mathcal{N}_{\ge 0}(\bar u),$
where $\bar \mu = D \bar u$. 

Now we show the rate of convergence. Suppose $\{u^k\}$ remain bounded and $u^*$ be one of its limit point, i.e., optimal to \eqref{modified_model}, and $F^*$ be its objective value. Let $\mathcal B(M)$ be a closed ball with a sufficiently large radius of $M$ so that it contains $\{u^k\}$ and thus all of its limit points. We define $C := \frac{5}{2}\sqrt{Nn}M+ \sqrt{Nn}\|\nabla F(u^*)\|\eta_1+ Nn \eta_1^2$ so that $C_k \le  C$ for any $k \in \mathcal K$. In addition, for any optimal Lagrange multiplier $\mu^*$ of \eqref{modified_model}, Problem \eqref{modified_model}, i.e., $\min_{w \ge 0, Dw = 0} F(w)$, is equivalent to $\min_{w \ge 0} F(w) + (\mu^*)^{\top} D w + \frac{1}{4\eta_k} \|Dw\|^2$. Therefore, we have
\begin{equation}
    F^* - Q(u^k; \eta_k) \le (\mu^*)^{\top} Du^k \le \|\mu^*\|\|Du^k\|\label{eq:opt-bound:2}
\end{equation} from the suboptimality of $u^k$ to \eqref{modified_model}, that is, $F^* \leq F(u^k) + (\mu^*)^{\top} D u^k + \frac{1}{4\eta_k} \|Du^k\|^2$.

Now, suppose the inner loop during the $k$-th outer loop has not terminated after $t$ number of inner iterations, i.e., $\|w^t - w^{t-1}\| > \eta_k^2$. Note that we have
{\footnotesize\begin{subequations}
    \begin{align}
& \|w^t - w^{t-1}\|^2  \nonumber \\
\le& \frac{4\eta_k(Q(u^{k-1}; \eta_k) - Q(w^t;\eta_k))}{t}\label{eq:telescope}\\
\le& \frac{4\eta_k(Q(u^{k-1}; \eta_k) - Q(u^k;\eta_k))}{t}\label{eq:inter}\\
 = & \frac{4\eta_k(Q(u^{k-1}; \eta_{k-1}) + \frac{1-r}{4r}\frac{1}{\eta_{k-1}}\|Du^{k-1}\|^2 - Q(u^{k};\eta_k))}{t} \label{eq:consecutive}\\
\le &\frac{4\eta_k(\frac{1}{r}\eta_k C_{k-1} + \frac{1-r}{4r}\frac{1}{\eta_{k-1}}\|Du^{k-1}\|^2 + \|\mu^*\|\|Du^k\|)}{t}\label{eq:opt-bound}\\
\le & \frac{4\eta_k\left(\frac{1}{r}\eta_k C_{k-1}  + \frac{1-r}{r^2}\eta_{k}(2\|\mu^*\|+\sqrt{C_{k-1}})^2 + 2\eta_k\|\mu^*\| \Big(2 \|\mu^*\| + \sqrt{C_k}\Big)\right)}{t}\label{eq:norm-bound}\\
\le & \frac{\overline M \eta_k^2}{t}, \mbox{ where $\overline M:= 4\left(\frac{1}{r}C + \Big( 2(\frac{1-r+r^2}{r^2})\|\mu^*\|+ \frac{1-r}{r^2}\sqrt{C}\Big) \Big(2 \|\mu^*\| + \sqrt{C}\Big)\right)$},
\end{align}\label{eqs:iteration-bound}
\end{subequations}}
where \eqref{eq:telescope} is obtained by telescoping the inequality in Lemma \ref{lemm:descent} for the $t$ number of inner iterations and using Lemma \ref{lemm:monotone}; \eqref{eq:inter} holds from Lemma \ref{lemm:descent} and the fact that $w^t$ is an intermediate inner-loop iterate; \eqref{eq:consecutive} holds since $Q(u^{k-1}; \eta_k) - Q(u^{k-1}; \eta_{k-1})= \frac{1}{4}(\frac{1}{\eta_{k}} - \frac{1}{\eta_{k-1}})\|Du^{k-1}\|^2$ and $\eta_k=r \eta_{k-1}$; \eqref{eq:opt-bound} is obtained by adding and subtracting $F^*$ and using \eqref{eq:opt-bound:1} and \eqref{eq:opt-bound:2}; \eqref{eq:norm-bound} follows from \eqref{eq:2}.

Thus, if $t>\frac{\overline M}{\eta_k^2}$, it would contradict the assumption that $\|w^t - w^{t-1}\| > \eta_k^2$. Therefore, the inner loop takes at most $\frac{\overline M}{\eta_k^2}$ iterations. 
In addition, to achieve $\eta_k = \eta_1 r^k = \epsilon$, we require $k = \frac{\log (\epsilon/\eta_1) }{\log r}$ outer loop iterations. As a result, the total number of executions of \eqref{firm} needed is $\sum_{k=0}^{\frac{\log (\epsilon/\eta_1)}{\log r}} \frac{\overline M}{\eta_1^2}r^{-2k} = \frac{\overline M}{\eta_1^2} \frac{1-r^{-2(\frac{\log (\epsilon/\eta_1)}{\log r} + 1)}}{1-r^{-2}} = \frac{\overline M}{\eta_1^2} \frac{1-r^{\frac{\log (\eta_1/\epsilon)^2}{\log r}}r^{-2}}{1-r^{-2}} = \frac{\overline M}{\eta_1^2} \frac{1-(\eta_1/\epsilon)^2 r^{-2}}{1-r^{-2}} = O(\frac{1}{\epsilon^2})$.
\end{proof}

We show that for the outer-loop iterations $k$ where $\eta_k \ge \max\{\|\mu^*\|, \sqrt{f(u^*)}, \|A^\top Db\|\}$, the iteration complexity of the inner loop improves to $O(1/\eta_k)$, as opposed to the previous bound of $O(1/\eta_k^2)$. Therefore, \texttt{FIRM} requires a total of $O(1/\epsilon)$ executions of \eqref{firm} to achieve an accuracy $\epsilon \ge \max\{\|\mu^*\|, \sqrt{f(u^*)}, \|A^\top Db\|\}$; specifically, note that $\sum_{k=0}^{\frac{\log(\epsilon/\eta_1)}{\log r}} \frac{\overline M}{\eta_1}r^{-k}=\frac{\overline M}{\eta_1} \frac{1-(\eta_1/\epsilon) r^{-1}}{1-r^{-1}}=O(1/\epsilon)$. To establish this, we first present the following technical result that is essential for proving the desired outcome, the proof of which is given in Appendix \ref{appendix:tech}:
\begin{lemma}
\begin{enumerate}
\item[(i)] Let $h(z) := \|z\|^2$. Then, for any $\delta > 0$ and $z^*$, $|h(z) - h(z^*)| \le \delta$ implies $\|z-z^*\| \le \sqrt{\delta + \|z^*\|^2} + \|z^*\|$.
\item[(ii)] Let $x$ and $y$ be nonnegative scalars, then $\sqrt{x+y}+\sqrt{x} \le 2\sqrt{x+\frac{y}{2}}$.
\end{enumerate}
\label{lemm:tech}
\end{lemma}

Building upon Theorem \ref{theo:convergence} and Lemma \ref{lemm:tech}, we establish the following:
\begin{theorem} \label{thm:convergence_rate}
 Let $(u^{*}, \mu^*)$ be any optimal primal-dual solution pair to \eqref{modified_model}, i.e., it satisfies \eqref{eq:KKT}. For the outer-loop iterations $k \in \mathbb Z_{\ge 0}$ with 
 \begin{equation}\eta_k \ge \max\{\|\mu^*\|, \sqrt{f(u^*)}, \|A^\top Db\|\},\label{eq:eta-bound}\end{equation} 
 the iteration complexity of the inner loop becomes $O(1/\eta_k)$. 
\end{theorem}
\begin{proof}
Let $\Omega:=\sqrt{Nn}\|\nabla F(u^*)\|+ Nn\eta_1$ and $\omega:=\frac{5}{2}\sqrt{Nn}$ so that $C_k \le \omega \|u^k-u^*\| + \eta_k \Omega$. First, note that when $\eta_k$ satisfies \eqref{eq:eta-bound}, we have
{\footnotesize
\begin{subequations}
    \begin{align}
 \|w^t - w^{t-1}\|^2  \le & \frac{4\eta_k\left(\frac{1}{r}\eta_k C_{k-1}  + \frac{1-r}{r^2}\eta_{k}(2\|\mu^*\|+\sqrt{C_{k-1}})^2 + 2\eta_k\|\mu^*\| \Big(2 \|\mu^*\| + \sqrt{C_k}\Big)\right)}{t}\label{eq:norm-bound2}\\
\le & \frac{4\eta_k\left(\frac{1}{r}\eta_k C_{k-1}  + \frac{1-r}{r^2}\eta_{k}(2\eta_k+\sqrt{C_{k-1}})^2 + 2\eta_k^2 \Big(2 \eta_k + \sqrt{C_k}\Big)\right)}{t},\label{eq:eta}
\end{align}
\end{subequations}}
where \eqref{eq:norm-bound2} is from \eqref{eq:norm-bound}, and \eqref{eq:eta} follows from the assumption $\eta_k \ge \|\mu^*\|$.

Note that if $\frac{\|u^{k}-u^*\|}{\eta_k}$ is upper bounded by some constant that is independent of $k$ for $k$ satisfying \eqref{eq:eta-bound}, then $C_k = \eta_k \texttt{c}$ for some constant \texttt{c}, and thus by using $\eta_k = r \eta_{k-1}$, we get:
{\footnotesize\begin{align*}
 \|w^t - w^{t-1}\|^2 \le \frac{4\eta_k^3\left(\frac{1}{r^2} \texttt{c}  + \frac{1-r}{r^2}(2\sqrt{\eta_k}+\frac{1}{\sqrt{r}}\sqrt{\texttt{c}})^2 + 2 \Big(2 \eta_k + \sqrt{\eta_{k}\texttt{c}}\Big)\right)}{t},
\end{align*}}
which yields
\begin{align}
\|w^t - w^{t-1}\|^2  \le \frac{\texttt{C}\eta_k^3}{t}, \label{bound_iterates}
\end{align}
where $\texttt{C}:= 4 \left[ \frac{1}{r^2}\texttt{c} + \frac{1-r}{r^2}\left(2\sqrt{\eta_1} + \frac{1}{\sqrt{r}}\sqrt{\texttt{c}}\right)^2 + 2\left(2\eta_1 + \sqrt{\eta_1}\sqrt{\texttt{c}}\right) \right]$. Thus, if $\frac{\|u^{k}-u^*\|}{\eta_k}$ is upper bounded by some constant that is independent of $k$ for $k$ with \eqref{eq:eta-bound}, the associated inner loops take at most $\frac{\texttt{C}}{\eta_k}$ iterations, and the desired result follows.

We show that $\frac{\|u^{k}-u^*\|}{\eta_k}$ is indeed upper bounded by a constant that is independent of $k$ for $k$ satisfying \eqref{eq:eta-bound}. First, note that from \eqref{eq:2} we have
\begin{equation}
\frac{\|Du^k\|}{\eta_k} \le 2(2\|\mu^*\| + \sqrt{C_k}) \le \texttt{a} \sqrt{\frac{C_k}{\eta_k}} + \texttt{b},
\label{eq:conv:2}
\end{equation}
for $\texttt{a} := 2\sqrt{\eta_1}$ and $\texttt{b} := 4\eta_1$ since $\|\mu^*\| \le \eta_k \le \eta_1$ and $\sqrt{C_k} = \sqrt{\eta_k}\sqrt{\frac{C_k}{\eta_k}} \le \sqrt{\eta_1}\sqrt{\frac{C_k}{\eta_k}}$. 

Note that, from \eqref{eq:opt-bound:1}, \eqref{eq:opt-bound:2}, and \eqref{eq:2}, we have 
\begin{align}|Q(u^{k}; \eta_k) - F(u^*)| & \le  \max\{\eta_kC_{k}, \|\mu^*\|\|Du^k\|\} \le \eta_kC_{k}+\|\mu^*\|\|Du^k\|,\label{eq:opt-bound:combined}
\end{align}
where the last inequality follows from $\max\{x,y\} \le x+y$ for nonnegative $x, y$, and
{\footnotesize\begin{subequations}
\begin{align}&|Q(u^{k}; \eta_k) - F(u^*)| \nonumber \\
& = \Big|\|\bar Au^{k} - b\|^2 -\frac{1}{2}\|ADu^{k} - Db\|^2 + \frac{1}{4\eta_k}\|Du^{k}\|^2 - \|\bar Au^* - b\|^2 + \frac{1}{2}\|Db\|^2\Big|\nonumber \\
& = \Big|\|\bar Au^{k} - b\|^2 - \frac{1}{2}\|ADu^{k}\|^2 + b^\top D^\top ADu^{k} + \frac{1}{4\eta_k}\|Du^{k}\|^2 - \|\bar Au^* - b\|^2\Big| \nonumber\\
& \ge \Big|\|\bar Au^{k} - b\|^2 + b^\top D^\top ADu^{k}  - \|\bar Au^* - b\|^2\Big|\label{eq:bound:1}\\
& \ge \Big|\|\bar Au^{k} - b\|^2 - \|\bar Au^* - b\|^2\Big| - \| A^\top D b\|\|Du^{k}\|,\label{eq:bound:2}
\end{align}
\end{subequations}}
where \eqref{eq:bound:1} follows from $\frac{1}{2\eta_k}D^\top D -  D^\top A^\top AD  = D^\top (\frac{1}{2\eta_k}I_n - A^\top A)D\succeq 0$ since $\frac{1}{2\eta_k} \ge \frac{1}{2\eta_1} = \frac{2\lambda_{\max}(A^\top A)}{2-\sum_{i=1}^{N-1}c_i^2} \ge \lambda_{\max}(A^\top A)$ for any $k \in \mathbb Z_{\ge 0}$; \eqref{eq:bound:2} is obtained from $|x-y| \ge |x| - |y|$ and the Cauchy-Schwarz inequality.

Combining \eqref{eq:opt-bound:combined} and \eqref{eq:bound:2}, we obtain
{\footnotesize\begin{align*}\Big|\|\bar Au^{k} - b\|^2 - \|\bar Au^* - b\|^2\Big| &  \le \eta_kC_{k} + (\|\mu^*\|+\|A^\top D b\|)\|Du^{k}\| \\
&\le \eta_k^2\left(\frac{C_{k}}{\eta_k} + 2\frac{\|Du^{k}\|}{\eta_k}\right) \\
&\le \eta_k^2\left(\frac{C_{k}}{\eta_k} + 2\texttt{a} \sqrt{\frac{C_{k}}{\eta_k}} + 2\texttt{b}\right),
\end{align*}}
where the second inequality follows from \eqref{eq:eta-bound} with the last inequality from \eqref{eq:conv:2}.

Then, Lemma \ref{lemm:tech}(i) yields
{\footnotesize\begin{align}\|\bar A(u^{k} - u^*)\| & \le \sqrt{\eta_k^2\left(\frac{C_{k}}{\eta_k} + 2\texttt{a} \sqrt{\frac{C_{k}}{\eta_k}} + 2\texttt{b}\right) + f(u^*)} + \sqrt{f(u^*)}\nonumber \\
& \le 2\sqrt{\eta_k^2\left(\frac{C_{k}}{2\eta_k} + \texttt{a} \sqrt{\frac{C_{k}}{\eta_k}} + \texttt{b}\right) + f(u^*)}\nonumber \\
& \le 2\eta_k\sqrt{\frac{C_{k}}{2\eta_k} + \texttt{a} \sqrt{\frac{C_{k}}{\eta_k}} + \underbrace{\texttt{b}+1}_{=:\texttt{b}'}},\label{eq:conv:3}
\end{align}}
where $f(u^*) = \|\bar Au^* - b\|^2$, the second inequality follows from Lemma \ref{lemm:tech} (ii), and the last inequality from \eqref{eq:eta-bound}. 

We now apply Hoffman's bound to estimate a bound on $\|u^k - u^*\|$ using the bound established for $\|A(u^{k} - u^*)\|$. Hoffman's bound provides a way to measure how close a point is to the feasible set of a system of linear inequalities, based on the constraint violations. In our case, the set of optimal solutions can be expressed as: $\mathcal W^* := \{w \in \mathbb R^{Nn}: w \ge 0, Dw = 0, \bar Aw = \bar Au^*\}$. By Hoffman's bound, the distance $\|u^k-u^*\|$, representing how close the point $u^k$ to the optimal solution set $\mathcal W^*$, can be bounded by the extent to which $u^k$ violates the linear inequalities defining $\mathcal W^*$. Specifically, there exists a constant $\gamma>0$ such that
\begin{equation}
\|u^k - u^*\| \le \gamma \left\|\Big((u^k)_-; Du^k; \bar A(u^k - u^*)\Big)\right\| = \gamma \sqrt{\|Du^k\|^2 + \|\bar A(u^k - u^*)\|^2},
\end{equation}
where the equality follows since $u^k \ge 0$. Therefore, we have
{\footnotesize\begin{align*}\frac{\|u^k - u^*\|}{\eta_k} &\le \gamma \sqrt{\left(\frac{\|Du^k\|}{\eta_k}\right)^2 + \left(\frac{\|\bar A(u^k - u^*)\|}{\eta_k}\right)^2}\\
&\le \gamma \sqrt{\left(\texttt{a} \sqrt{\frac{C_{k}}{\eta_k}} + \texttt{b}\right)^2 + 2\frac{C_k}{\eta_k}+ 4\texttt{a} \sqrt{\frac{C_k}{\eta_k}} + 4\texttt{b}'}\\
& =  \gamma \sqrt{(\texttt{a}^2+2)\frac{C_{k}}{\eta_k} + (2\texttt{ab} + 4\texttt{a})\sqrt{\frac{C_{k}}{\eta_k}} + \texttt{b}^2 + 4\texttt{b}'}\\
& \le \gamma \sqrt{(\texttt{a}^2+2) \omega \frac{\|u^k-u^*\|}{\eta_k}  + (2\texttt{ab}+4\texttt{a}) \sqrt{\omega \frac{\|u^k-u^*\|}{\eta_k}}+\texttt{const}},
\end{align*}}
where \(\texttt{const} =  (\texttt{a}^2+2)\Omega+(2\texttt{ab}+4\texttt{a})\sqrt{\Omega} + \texttt{b}^2 + 4\texttt{b}'\), the second inequality follows from \eqref{eq:conv:2} and \eqref{eq:conv:3}, and the last inequality from $\sqrt{x+y} \le \sqrt{x} + \sqrt{y}$ for nonnegative scalars $x$ and $y$. Note that 
the resultant inequality is violated for any sufficiently large value of $\frac{\|u^k - u^*\|}{\eta_k}$. Hence, the desired result follows.
\end{proof}

\begin{remark}The measurements \( b \) often include noise, e.g., \( b = b_{\texttt{true}} + \delta \), where \( \delta \) represents the realized noise, perturbing the true measurements \( b_{\texttt{true}} \).
 Notably, the bound $\max\{\|\mu^*\|, \sqrt{f(u^*)}, \|A^\top Db\|\}$ can be quite small for low-noise instances, i.e., $b \approx b_{\texttt{true}}$. Accurate measurements $b$ reduce the cost of satisfying the equality constraint, as the image that minimizes the associated least squares naturally satisfies the equality constraint. This makes $\mu^*$ small. Additionally, lower noise reduces discrepancies in measurements, leading to smaller $Db$ since $Db_{\texttt{true}}=0$, smaller least squares residuals, and a smaller $f(u^*)$. This implies that \texttt{FIRM} will achieve a $\epsilon$-optimality with an iteration complexity of $O(1/\epsilon)$ for a reasonable accuracy $\epsilon$ with low noise. As noise increases, however, the bound grows, which can negatively affect the convergence rate of \texttt{FIRM} for achieving high accuracy. Nonetheless, solving \eqref{basic_model} to high precision is often unnecessary in noisy cases, meaning \texttt{FIRM} can still perform effectively.
    \label{rema:eps-convergence}
\end{remark}

\begin{remark}
    Recent developments in first-order penalty methods for general convex programming problems have yielded some relevant convergence analyses \cite{lan2013iteration,necoara2019complexity,li2020decentralized}. Of particular relevance is \cite{lan2013iteration}, where the accelerated Nesterov's method is employed to solve the penalized problem. This work establishes an iteration complexity of \(O(1/\epsilon^2)\) for the first-order penalty method applied to the original problem, with a penalty parameter updated via a guess-and-check framework. In this paper, we prove that the same convergence rate can be achieved using a simple geometric penalty parameter update rule, combined with a problem-tailored inner-loop termination criterion, as outlined in Algorithm \ref{algo}. Additionally, we provide conditions under which the algorithm achieves an improved \(O(1/\epsilon)\) rate of convergence. Another relevant contribution is \cite{necoara2019complexity}, which demonstrates that one side of the optimality gap can be closed within \(O(1/\epsilon^{1.5})\) iterations, specifically for achieving \(F(u^k) - F(u^*) \leq \epsilon\) with \(\|Du^k\| \leq \epsilon\), providing only an \(\epsilon\)-asymptotic lower bound to the original optimal objective value. In contrast, our work presents the convergence rate required to approach the true optimal solution satisfying the KKT conditions, without needing prior knowledge of the optimal Lagrange multiplier \(\mu^*\). Lastly, \cite{li2020decentralized} analyzed a first-order accelerated penalty method with an iteration complexity of \(O(1/\epsilon^2)\).
    \label{rema:convergence-results}
\end{remark}

\subsection{\texttt{FIRM}$^+$: An Augmented Lagrangian Variant of \texttt{FIRM}} \label{sec:augmented_lagrangian}
The interpretation of \texttt{FIRM} as an inexact first-order QP method for solving \eqref{modified_model} naturally suggests its extension to an inexact first-order Augmented Lagrangian (AL) method to leverage the favorable properties of AL methods.
We refer to this extended method as \texttt{FIRM}$^+$. In this section, we outline the extension of \texttt{FIRM} to \texttt{FIRM}$^+$, deferring a detailed analysis to future research. This can be achieved by introducing Lagrange multipliers \( \mu \) for the penalized constraints \( D w = 0 \).  Specifically, AL method works with a modified penalty function, commonly called the augmented Lagrangian function,
$\mathcal{L}_{A} (w, \mu; \eta) :=  Q(w; \eta)    + \langle \mu, Dw \rangle,$
which adds a Lagrangian term of $\langle \mu, Dw \rangle$ to the quadratic penalty function $Q$, where $\mu \in \mathbb{R}^{n}$ represents the dual variables associated with the constraint $Dw=0$ in \eqref{modified_model}.

AL method solves a sequence of problems $\{\min_{w \ge 0}\mathcal{L}_{A} (w, \mu^k; \eta_k)\}_k$, while updating $\mu^k$ according to the rule:
\begin{align}
& \mu^{k+1} \gets \mu^k + \frac{1}{2\eta_k} Du^{k}, \label{dual_update}
\end{align}
where $u^k$ is an approximate solution to $\min_{w \ge 0}\mathcal{L}_{A} (w, \mu^k; \eta_k)$.

In our extension, each $\min_{w \ge 0}\mathcal{L}_{A} (w, \mu^k; \eta_k)$ is solved approximately via several PGD steps. Note that the update on $w$ after one PGD step from the current iterate $w^t$ can be expressed as follows, based on \eqref{firm:equiv}:
$$w^{t+1} \gets \Big(w^t - \eta_k g^t - \frac{1}{2} D^\top D v^t - \eta_kD^\top \mu^k\Big)_+,$$
which suggests a natural modification to \eqref{our_approach_update}, leading to the following: 
\begin{align}
& z^t_i \gets
\begin{cases}
v^t_i + c_i (x^t - y^t) + \eta_k c_i \mu^k & \text{ for } i \in [N-1], \\
x^t  - \eta_k \mu^k & \text{ for } i = N.
\end{cases}  \label{al_update}
\end{align}

Thus, an extension \texttt{FIRM}$^+$ is obtained by replacing \eqref{our_approach_update} in Line \ref{line:server} with \eqref{al_update}, and by incorporating the $\mu$-update step \eqref{dual_update} just before Line \ref{line:update}.


\section{Experimental Results}
\label{sec:experiments}
In this section, we numerically demonstrate the performance of \texttt{FIRM} for solving \eqref{basic_model} and the benefits of incorporating multimodality on solution quality by comparing the results of \eqref{basic_model} to those of the individual image reconstruction (IIR) model. All numerical tests were performed on a workstation composed of (i) NVIDIA RTX 6000 Ada Generation, 48 GB GDDR6 GPU, and (ii) 56 cores with Intel Xeon w9-3495X processors and 512 GB DDR5 of memory.

\subsection{Data} \label{sec:experimental_data}
\begin{figure}[t!]
    \centering
    \begin{subfigure}[b]{0.24\textwidth}
    \centering
    \includegraphics[width=\textwidth,trim={2cm 0.5cm 2cm 1cm},clip]{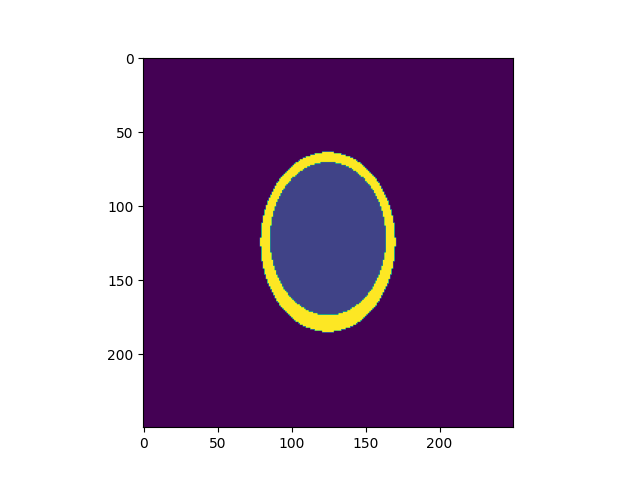} 
    \caption{XRF-1}  
    \end{subfigure}
    \begin{subfigure}[b]{0.24\textwidth}
    \centering
    \includegraphics[width=\textwidth,trim={2cm 0.5cm 2cm 1cm},clip]{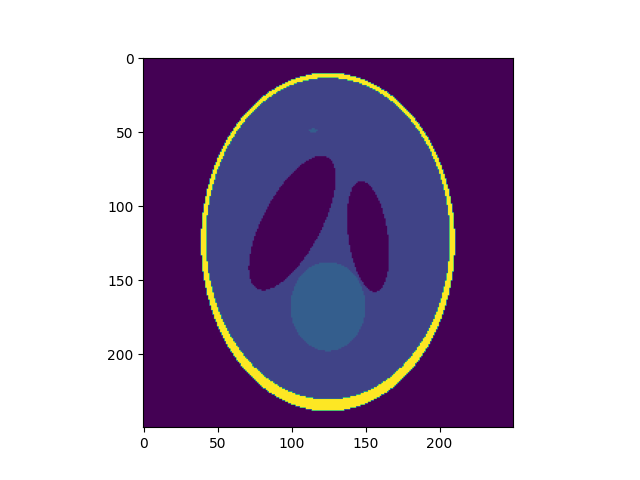}
    \caption{XRF-2}  
    \end{subfigure}
    \begin{subfigure}[b]{0.24\textwidth}
    \centering
    \includegraphics[width=\textwidth,trim={2cm 0.5cm 2cm 1cm},clip]{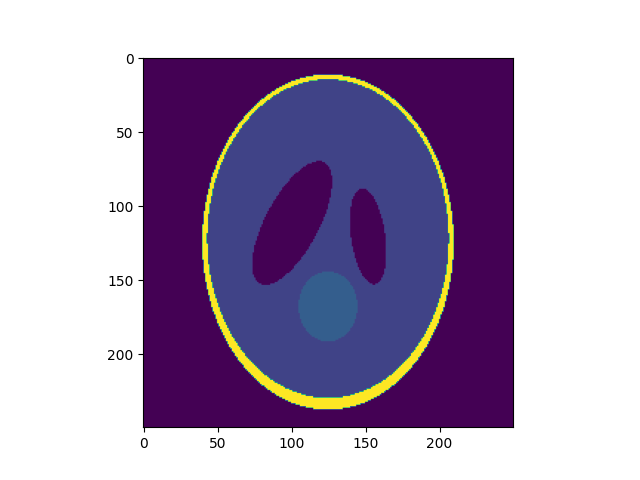} 
    \caption{XRF-3}  
    \end{subfigure}
    \begin{subfigure}[b]{0.24\textwidth}
    \centering
    \includegraphics[width=\textwidth,trim={2cm 0.5cm 2cm 1cm},clip]{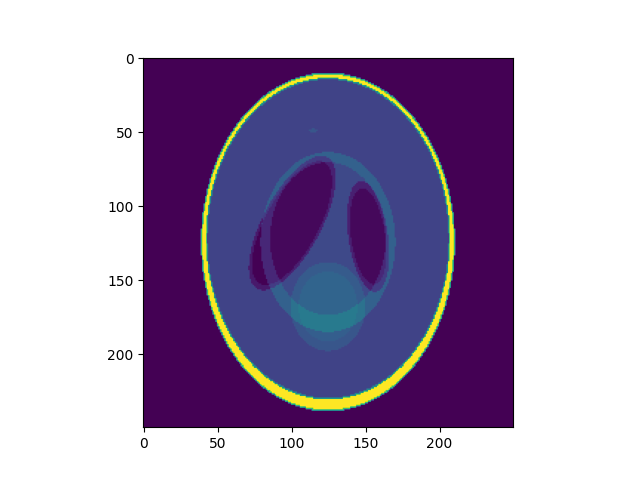} 
    \caption{XRT} 
    \end{subfigure}
    \caption{The ground-truth images.}
    \label{fig:ground_truth} 
    \end{figure}
    We use synthetic samples derived from variants of the Shepp-Logan phantom, referred to as XRF-1, XRF-2, XRF-3, and XRT (see Fig.~\ref{fig:ground_truth}). The index \(i \in [N] := [4]\) corresponds to XRF-1, XRF-2, XRF-3, or XRT in the given order. These ground-truth images, each of size \(250 \times 250\), are represented in vectorized form as \(w_i^{\text{true}} \in \mathbb{R}^n\), where \(n = 250^2\) and \(i \in [N]\). 

    For a varying number of angles \( |\Theta| \in \mathcal{N}_{\theta} := \{25, 50, 75, 100, 125, 150, 175\} \), we generate the corresponding tomographic dataset \(\mathcal{B}_{|\Theta|}\) to evaluate performance across different data quantities while keeping the number of beamlets fixed at \( |\mathcal{T}| = \lceil \sqrt{2n} \rceil = 354 \). 
    This results in datasets of the form:
\(
\mathcal{B}_{|\Theta|} = \{b_i \in \mathbb{R}^{|\mathcal{T}| |\Theta|}, \forall i \in [N]\} 
\)
for \( |\Theta| \in \mathcal{N}_{\theta} \). The corresponding discrete Radon transform \( A \) is adjusted based on the value of $|\Theta|$ such that \( A \in \mathbb{R}^{|\Theta| |\mathcal{T}| \times n} \). In practice, tomographic datasets may involve noise originating from experimental processes.
To capture this, we generate random noisy tomographic datasets \(\widetilde{\mathcal{B}}_{|\Theta|,s}\) by adding Gaussian noise with zero mean and standard deviation \(s \in \mathcal{S} := \{0, 0.01, 0.05, 0.1\}\) to each vector \(b \in \mathcal{B}_{|\Theta|}\). For each $s \in \mathcal S \setminus \{0\}$, this process results in the random noisy dataset \(\widetilde{\mathcal{B}}_{|\Theta|,s}\), while for $s=0$, we have $\widetilde{\mathcal{B}}_{|\Theta|,0} = \mathcal B_{|\Theta|}$, i.e., a noise-free deterministic dataset. As \(s\) increases, \(\widetilde{\mathcal B}_{|\Theta|,s}\) becomes progressively noisier. We provide visual examples of the tomographic datasets in Appendix~\ref{apx:data}. Finally, the coefficients of the linear combination constraints are given as  \(
c = [0.1 \ \ 0.6 \ \ 0.3]^\top.
\)

\subsection{Performance of \texttt{FIRM}} \label{sec:algorithm_comparison}

\begin{table}[b!]
    \centering
    \footnotesize
    \caption{The value of $\tau^s:=\max\{\|\mu^*\|, \sqrt{f^*}, \|A^{\top} Db^s\| \}$}
    \begin{tabular}{|l| r r r r|} 
    \hline
           $s$     & $0$ & $0.01$ & $0.05$ & $0.1$  \\ \hline        
    $\tau^s$ & $9.44 \times 10^{-5}$ &   $1.18\times 10^{-2}$ &   $5.55\times 10^{-2}$ &   $1.11\times 10^{-1}$ \\ \hline
    \end{tabular}
    \label{table:each_element}
    \end{table}
We numerically demonstrate the performance of \texttt{FIRM} for solving \eqref{basic_model} with a termination tolerance of $\epsilon=10^{-2}$. In this section, we utilize a single sample \(\mathcal{B}_{25,s}\) of the random tomographic dataset \(\widetilde{\mathcal{B}}_{25,s}\) for \(s \in \mathcal{S}\). Let \(b^s\) denote the concatenated vector obtained by appending the \(b_i\)'s in \(\mathcal{B}_{25,s}\). As shown in Theorem \ref{thm:convergence_rate}, the convergence rate of \texttt{FIRM} is given by \(\mathcal{O}(1/\epsilon)\) for \(\epsilon \geq \tau^s := \max\{\|\mu^*\|, \sqrt{f^*}, \|A^{\top} Db^s\|\}\). In Table \ref{table:each_element}, we present the values of \(\tau^s\), which are obtained by approximating \(\|\mu^*\|\) and \(\sqrt{f^*}\) through the solution of \eqref{basic_model}, computed using \texttt{FIRM+} for each dataset $\mathcal B_{25,s}$. As noted in Remark \ref{rema:eps-convergence}, \(\tau^s\) can be small in the noise-free case (i.e., when \(s=0\)), but it tends to increase as the noise level \(s\) increases. This indicates that \texttt{FIRM} can achieve an improved convergence rate of \(\mathcal{O}(1/\epsilon)\) for fairly small \(\epsilon\) values when the noise level is low. On the other hand, as \(s \in \mathcal{S}\) increases, \(\tau^s\) also grows.
A larger \(\tau^s\) may negatively affect the convergence rate of \texttt{FIRM} in achieving high accuracy. However, obtaining the lowest possible objective value in \eqref{basic_model} under noisy conditions is often unnecessary, as solutions with lower objective values may deviate more from the ground truth due to the noise (see Figure \ref{fig:ratio_noise}). This suggests that \texttt{FIRM} can still perform effectively in noisy scenarios.

\subsubsection{Effect of the ratio parameter}\label{sec:effect-r}
 
\begin{figure}[t!]
    \centering
    \begin{subfigure}[b]{0.25\textwidth}
    \centering
    \includegraphics[width=\textwidth]{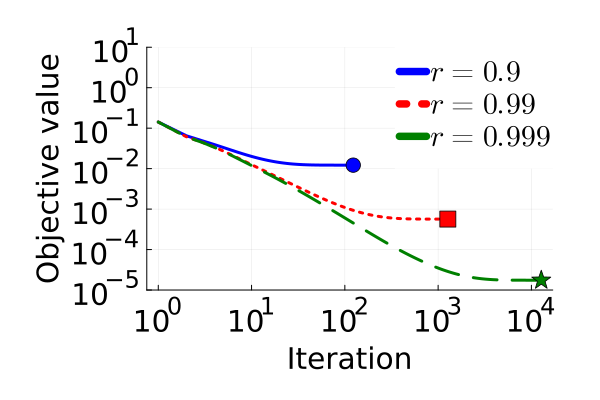}  
    \includegraphics[width=\textwidth]{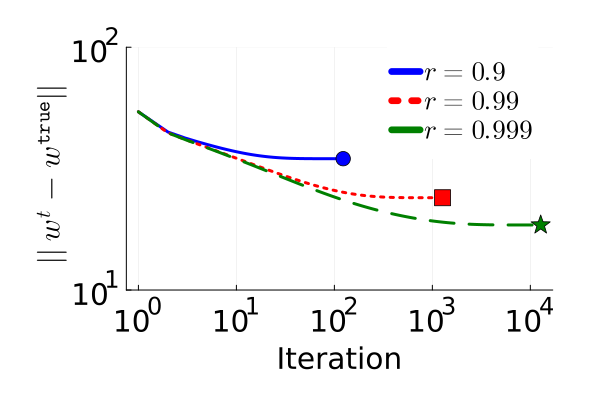}
    \caption{$s=0$}  
    \label{fig:ratio_no_noise}
    \end{subfigure}
    \begin{subfigure}[b]{0.25\textwidth}
    \centering
    \includegraphics[width=\textwidth]{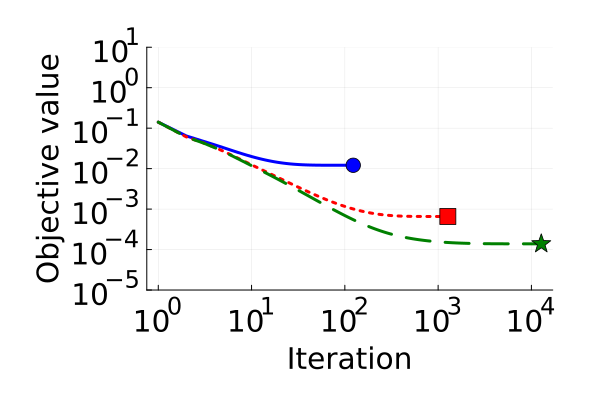}  
    \includegraphics[width=\textwidth]{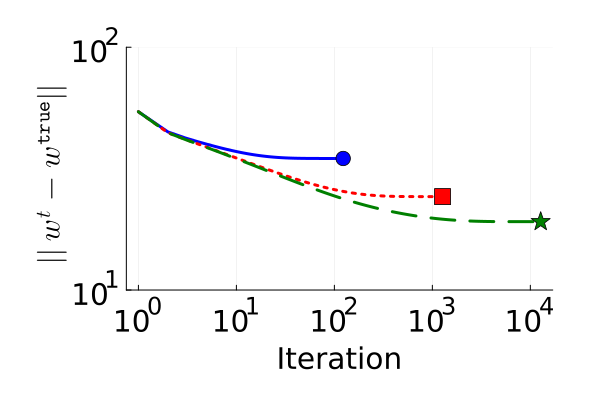}
    \caption{$s=0.01$}  
    \end{subfigure}
    \begin{subfigure}[b]{0.25\textwidth}
    \centering
    \includegraphics[width=\textwidth]{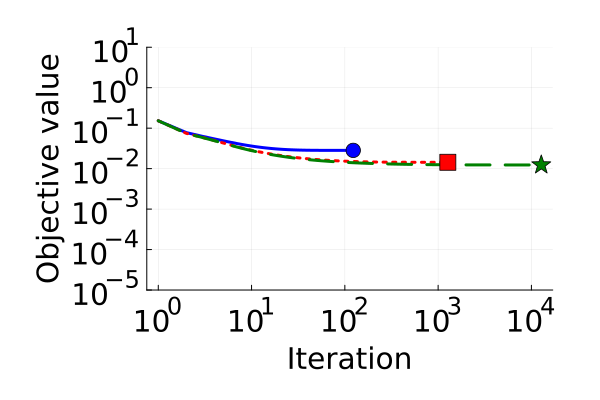}  
    \includegraphics[width=\textwidth]{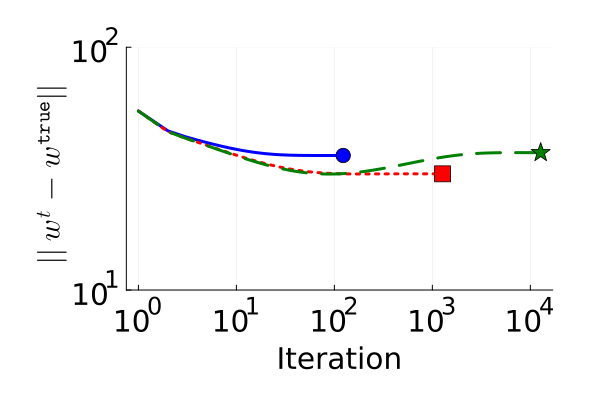}
    \caption{$s=0.1$}  
    \label{fig:ratio_noise}
    \end{subfigure}
    \caption{The objective function value $f(w^{t})$ and the distance from the current iterate to the ground-truth $\|w^{t} - w^{\texttt{true}}\|$ for every iteration $t \in \mathbb{Z}_{\geq 0}$. Note that both x- and y-axis are log scaled.}
    \label{fig:effect_of_ratio}
    \end{figure}

    The ratio parameter $r$ in \texttt{FIRM} determines how fast the step size $\eta_k$ decreases, i.e., the penalty parameter $0.25/\eta_k$ increases, where $k \in \mathbb{Z}_{\geq 0}$ represents the outer-loop iteration. 
    Figure \ref{fig:effect_of_ratio} illustrates the objective function value \( f(w^{t}) \) and the distance from the current iterate to the ground truth \( \|w^{t} - w^{\text{true}}\| \) for each iteration \( t \in \mathbb{Z}_{\geq 0} \), where all inner loop iterates are appended throughout the entire algorithm run, under various values of \( r \in \mathcal{R} := \{0.9, 0.99, 0.999\} \) and \( s \in \mathcal{S} \). Regardless of the value of \( r \), \texttt{FIRM} produces an \( \epsilon \)-optimal solution to \( (P_\epsilon) \) upon termination, according to Definition \ref{def:sub-opt-penalty}. However, as shown in Figure \ref{fig:ratio_no_noise}, the solutions obtained for different values of \( r \) are significantly different in terms of the objective function values. This observation aligns with the ill-conditioned nature of the problem, which has many optimal solutions, and suggests that the objective contour may be very flat, meaning the gradient is relatively small, even when the solution is far from optimal.

    This result is consistent with the theoretical analysis, particularly Proposition \ref{prop:opt-gap}. The optimality gap of the \( \epsilon \)-optimal solution to \( (P_{\epsilon}) \) in terms of the original problem depends on the distance from the current solution to an optimal solution (see \eqref{eq:opt-bound:1}). Therefore, if the problem is ill-conditioned, making the $\epsilon$-optimal solution to $(P_\epsilon)$ far from \( u^* \), we may still encounter a larger optimality gap with respect to the original problem. In such cases, a small stepsize can hinder the iterate from escaping the suboptimal region because the gradient can be quite small, even when the solution is far from optimal.
    This demonstrates that rapidly reducing the step size \( \eta_k \) with smaller \( r \) can cause the algorithm to stagnate in a flat, suboptimal region due to the problem's inherently ill-conditioned nature. 
    
    To avoid this issue, we maintain \( r = 1 \) for a fixed number of initial iterations and then gradually decrease the step size, e.g., at a rate of \( r = 0.9 \), to ensure constraint satisfaction. The results are \emph{remarkable}, showing that \texttt{FIRM} converges at the same rate as \texttt{FedPGD} when the step size is kept constant in the first phase, and in the second phase, it rapidly meets \texttt{FIRM}'s termination criterion, which will be detailed in the subsequent section.
    
Additionally, when noise is present in the dataset, Figure \ref{fig:ratio_noise} illustrates that solutions with smaller objective values deviate more from the true image. Consequently, reducing the optimality gap does not improve the outcome due to data noise. This observation motivates a revision of the termination criterion for noisy datasets, which we will discuss in a later section.


\subsubsection{Faster computation and constraint satisfaction} \label{sec:firm_fedpgd}
\texttt{FIRM} replaces the projection step \eqref{pgd_2} of \texttt{FedPGD} with a set of vector operations as in \eqref{our_approach}.
To numerically demonstrate the computational efficiency of \texttt{FIRM} compared to \texttt{FedPGD}, we apply \texttt{FedPGD} to solve \eqref{basic_model} on the same datasets. For \texttt{FedPGD}, we use a constant step size of $\alpha \gets 3/(4 \lambda_{\max}(A^{\top}A))$ 
and employ a Gurobi solver to perform the projection step \eqref{pgd_2}. The termination criterion is set with a tolerance of $\epsilon=10^{-2}$, stopping the algorithm when $\|w^{t} - w^{t-1}\| \leq \epsilon$.
For \texttt{FIRM}, we initialize the step size as $\eta_1 = \alpha$ and keep $r = 1$ (i.e., constant step size) for the first $10^4$ iterations, motivated by the discussion in \S\ref{sec:effect-r}. Subsequently, we set $r=0.9$ to gradually decrease the step size (i.e., increasing penalty parameter) to ensure constraint satisfaction.
 
\begin{figure}[t!]
\centering
\begin{subfigure}[b]{\textwidth}
\centering
\includegraphics[width=0.24\textwidth]{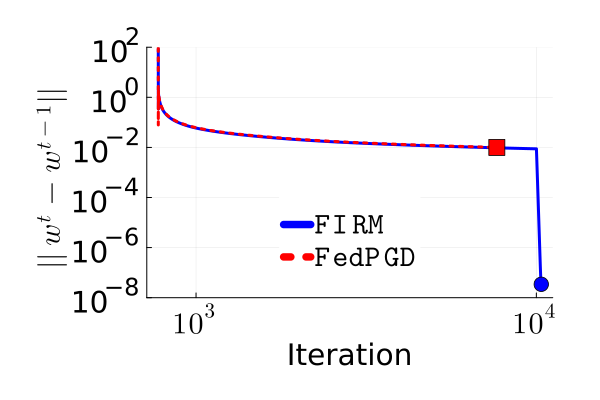}  
\includegraphics[width=0.24\textwidth]{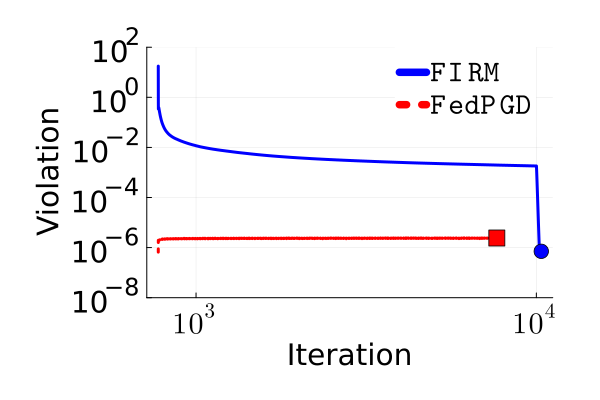}   
\includegraphics[width=0.24\textwidth]{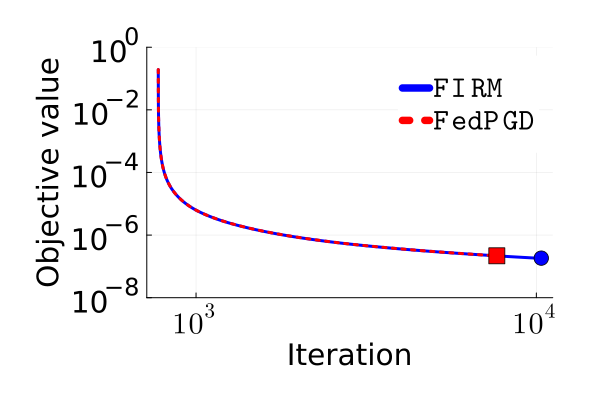}  
\includegraphics[width=0.24\textwidth]{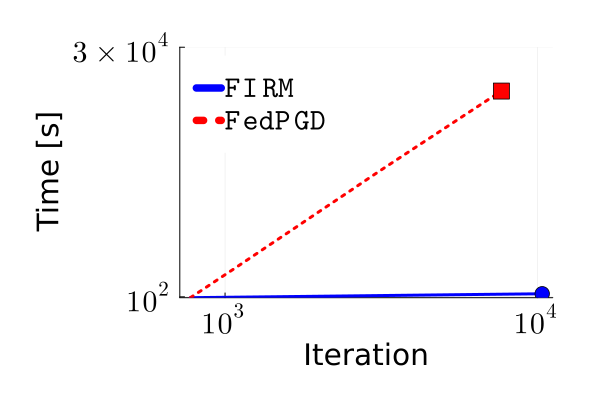}  
\caption{$s=0$}  
\label{fig:FIRM_FedPGD_0}
\end{subfigure} 
\hspace{5mm}
\begin{subfigure}[b]{\textwidth}
\centering
\includegraphics[width=0.24\textwidth]{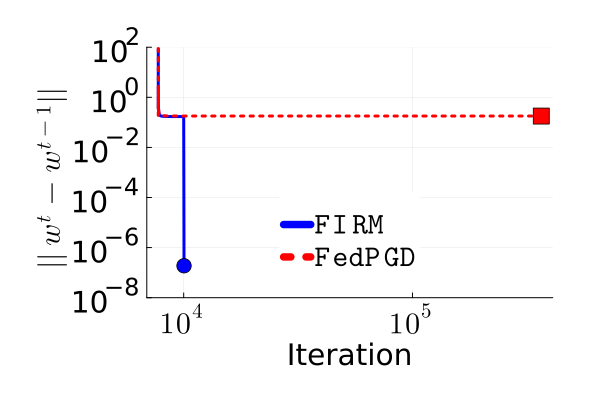}  
\includegraphics[width=0.24\textwidth]{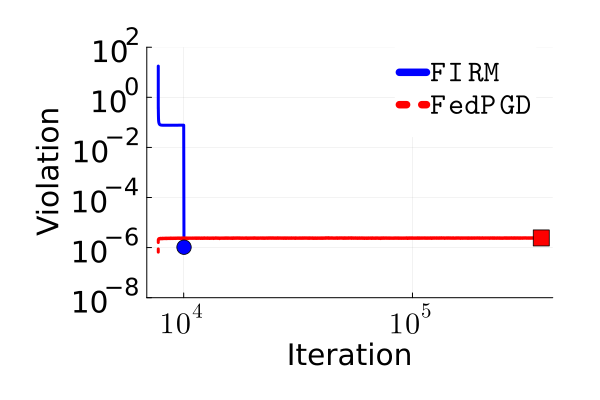}   
\includegraphics[width=0.24\textwidth]{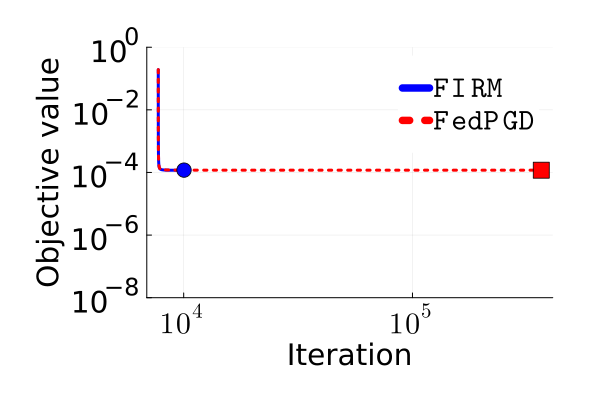}  
\includegraphics[width=0.24\textwidth]{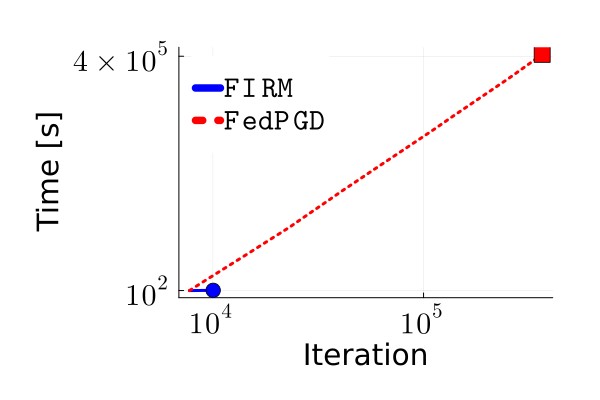}  
\caption{$s=0.01$} 
\label{fig:FIRM_FedPGD_1}
\end{subfigure}
\caption{The objective function value $f(w^{t}$), the distance between the consecutive iterates $\|w^t -w^{t-1}\|$, constraint violation $\|Dw^{t}\|$, and the elapsed time in seconds.}
\label{fig:FIRM_FedPGD}
\end{figure}


Figure \ref{fig:FIRM_FedPGD} 
compares \texttt{FIRM} and \texttt{FedPGD} in terms of the objective function value $f(w^{t}$), the distance between the consecutive iterates $\|w^t -w^{t-1}\|$, constraint violation $\|Dw^{t}\|$, and the elapsed time (in seconds) up to iteration $t \in \mathbb{Z}_{\geq 0}$. For \texttt{FIRM}, all inner-loop iterates are appended throughout the entire algorithm run. When $s = 0$, \texttt{FedPGD} terminates at iteration $8955$ while \texttt{FIRM} terminates at iteration $10130$ as shown in the first row of Figure \ref{fig:FIRM_FedPGD_0}. Notably, the convergence behavior of \texttt{FIRM} resembles that of \texttt{FedPGD} when $r=1$ (i.e., when both use the same step size). 
When $t > 10^4$, $\eta_k$ diminishes with $r=0.9$, which reduces (i) $\|w^t - w^{t-1}\|$ (see \eqref{bound_iterates} and the first row of Figure \ref{fig:FIRM_FedPGD_0}); 
and (ii) the constraint violation $\|Dw^t\|$ (see \eqref{eq:conv:2} and the second row of Figure \ref{fig:FIRM_FedPGD_0}).
While the objective function values generated by both algorithms are close to each other for every iteration $t \in \mathbb{Z}_{\geq 0}$, the computational performance gap is significant.
The elapsed computation times for \texttt{FedPGD} and \texttt{FIRM} are $24753$ and $469$ seconds, respectively, leading to approximately $52$ times faster computation. 

Second, when $s = 0.01$, \texttt{FedPGD} fails to terminate within several days, while \texttt{FIRM} successfully terminates, despite both methods achieving similar objective function values and constraint violations. This discrepancy arises because, after a certain number of iterations, the distance between consecutive iterates in \texttt{FedPGD} shows minimal improvement, failing to satisfy the termination criterion $\|w^t-w^{t-1}\| \le \epsilon = 0.01$, as shown in the first row of Figure \ref{fig:FIRM_FedPGD_1}, even though a fairly good objective function value has been reached. We attribute this behavior to the ill-conditioned nature of the problem with noise. In contrast, \texttt{FIRM} successfully terminates with a feasible solution in a much shorter computation time.

\subsection{Effect of multimodality} 
In this section, we numerically demonstrate the effectiveness of multimodality constraints in improving the quality of the reconstructed images. Specifically, we compare the solution produced by \eqref{basic_model} with those obtained from the individual image reconstruction (IIR) model for each 10 samples, \({\mathcal{B}}^1_{|\Theta|,s},\cdots, {\mathcal{B}}^{10}_{|\Theta|,s}\), of the random dataset \(\widetilde{\mathcal{B}}_{|\Theta|,s}\). Specifically, for $\ell \in [10],$ we let \(\widehat{w}^{|\Theta|,s,\ell}\) denote the solution of \eqref{basic_model} for ${\mathcal{B}}^\ell_{|\Theta|,s}$, and the solution of the IIR model is given by 

{\footnotesize\begin{align}
\widehat{w}^{\text{\tiny IIR}, |\Theta|, s,\ell}_i \gets \argmin_{w \in \mathbb{R}^n_+} f_i^{|\Theta|,s}(w) := \| A w - b_i \|^2, \ \forall i \in [N],\ell \in [10], b_i \in \mathcal B^\ell_{|\Theta|,s}.\label{model_individual}
\end{align}}

We use the projected gradient method, denoted by \texttt{PGD}, to solve the IIR model, which outperforms the other solution approaches in both computational efficiency and solution quality for the datasets (see Appendix \ref{apx:iir} for more details). When noise is present in the data (i.e., $s > 0$), the termination criteria used in \S\ref{sec:firm_fedpgd} are no longer effective in finding the ground-truth, as shown in Figure \ref{fig:ratio_noise}.
Therefore, we apply the \textit{discrepancy principle}, a widely used posteriori-stopping rule \cite{morozov1966solution} for cases where noise is present in the dataset.
Specifically, we set the termination criterion as:

{\footnotesize\begin{align}
& \| A w - b_i\| \leq \max \{b_{i}\} \sqrt{|\Theta||\mathcal T|} s, \ \forall i \in [N],
\end{align}}
where $\max \{b \}$ represents the largest element of the vector $b$. With this criterion, algorithms are terminated before reaching the optimal solution for the noisy model, so the choice of the step size may affect the quality of the solutions $\widehat{w}$, measured by $\| \widehat{w} - w^{\texttt{true}} \|$. 
To evaluate the effectiveness of multimodality across different settings, we experiment with varying step sizes. For \texttt{PGD}, we use a constant step size $\gamma \alpha$, where $\alpha \gets 3/(4 \lambda_{\max}(A^{\top}A))$ and $\gamma \in \Gamma:=\{0.02, 0.05, 0.1, 1\}$, to assess the impact of small and large step sizes. For \texttt{FIRM}, we use the same step size rule as in \S\ref{sec:firm_fedpgd} where the initial step size is set as $\eta_1 \gets \gamma \alpha$.
For the noise-free dataset (i.e., $s=0$), we use the termination criteria in \S\ref{sec:firm_fedpgd} with $\gamma = 1$.
 
\begin{figure}[!tbp]
\centering
\begin{subfigure}[b]{0.25\textwidth}
\centering     
\includegraphics[width=\textwidth]{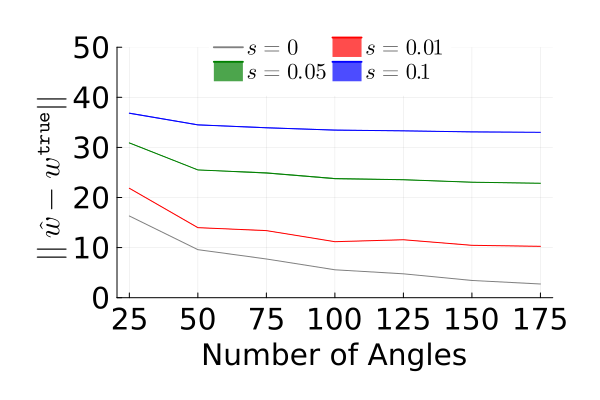}  
\includegraphics[width=\textwidth]{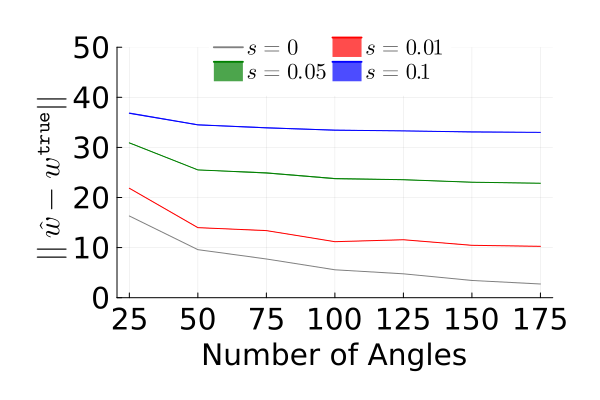}  
\includegraphics[width=\textwidth]{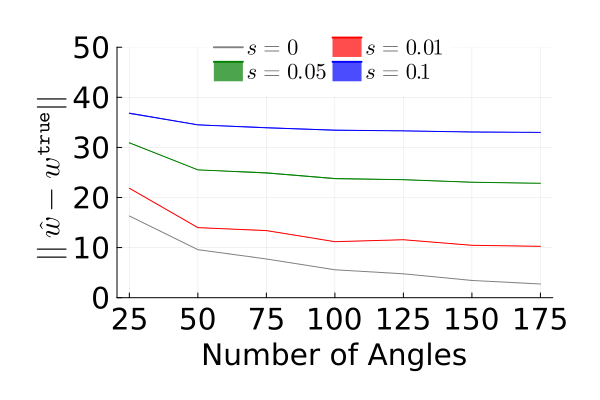}  
\includegraphics[width=\textwidth]{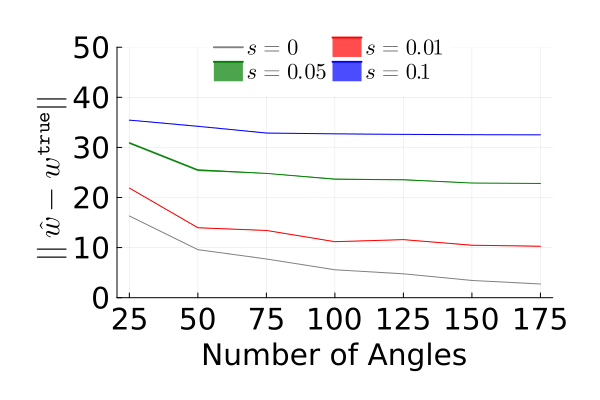}  
\caption{\eqref{basic_model} solved by \texttt{FIRM}}  
\label{fig:firm_relax_1} 
\end{subfigure}
\begin{subfigure}[b]{0.25\textwidth}
\centering     
\includegraphics[width=\textwidth]{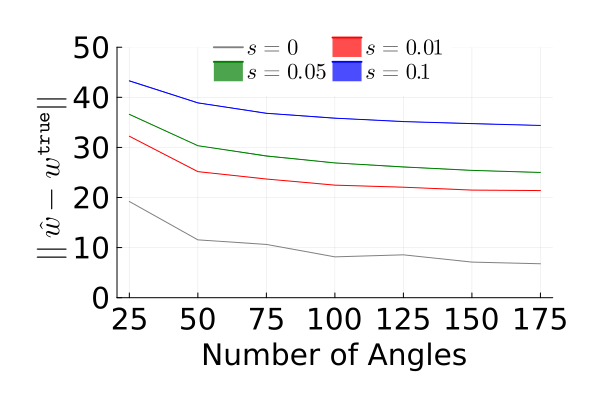}
\includegraphics[width=\textwidth]{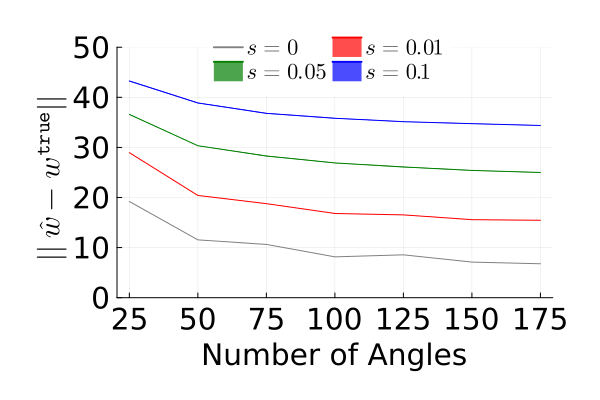}
\includegraphics[width=\textwidth]{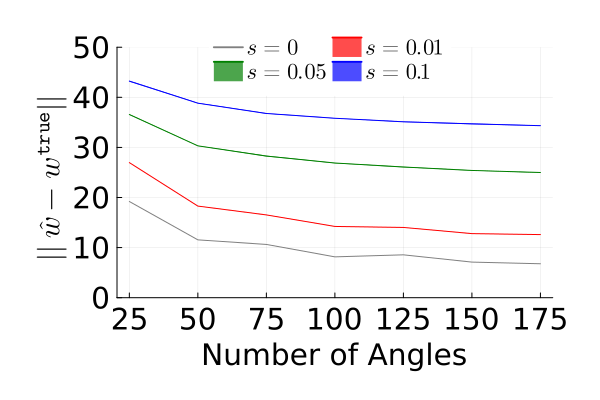}
\includegraphics[width=\textwidth]{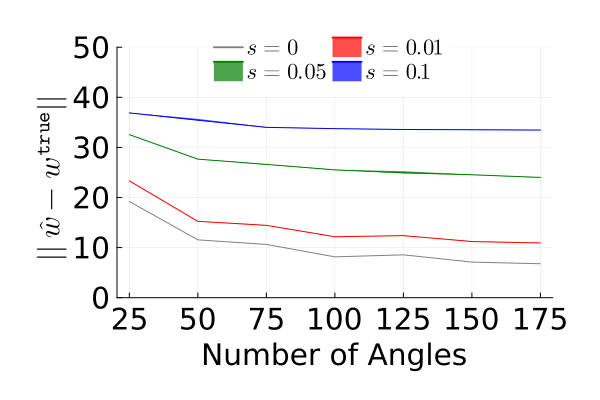}
\caption{IIR solved by \texttt{PGD}}  
\label{fig:firm_relax_2} 
\end{subfigure} 
\begin{subfigure}[b]{0.25\textwidth}
\centering     
\includegraphics[width=\textwidth]{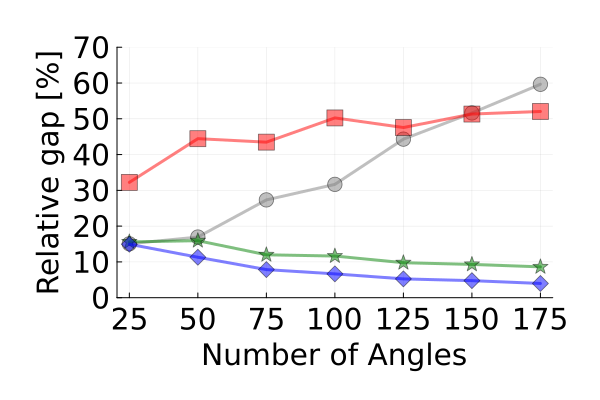}
\includegraphics[width=\textwidth]{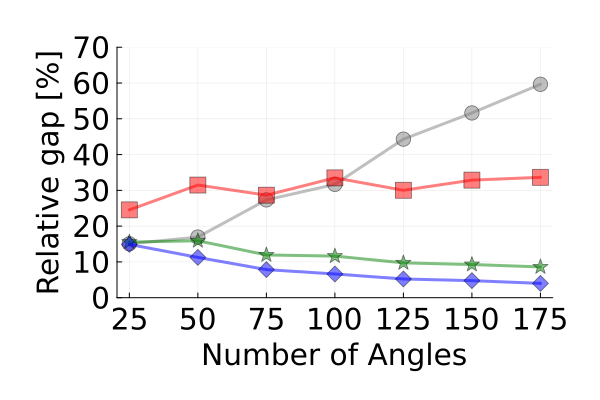}
\includegraphics[width=\textwidth]{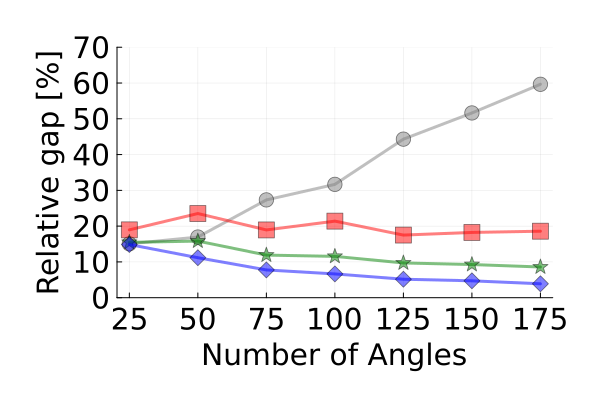}
\includegraphics[width=\textwidth]{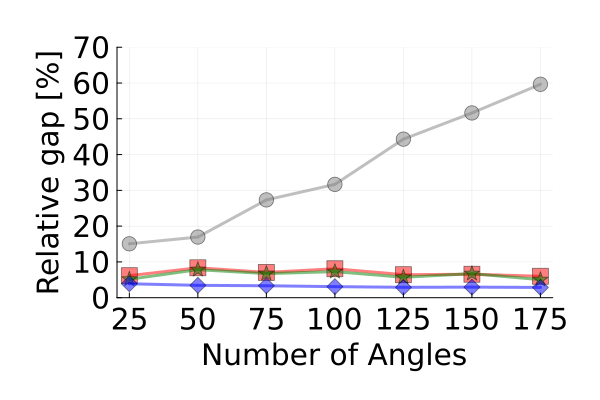}
\caption{Relative gap} 
\label{fig:firm_relax_3} 
\end{subfigure} 
\caption{Comparison of the quality of $\widehat{w}$, measured by $\|\widehat{w} - w^{\texttt{true}} \|$, where $\widehat{w}$ in (a) and (b) are obtained by solving \eqref{basic_model} via \texttt{FIRM}, and the IIR model via \texttt{PGD}, respectively. The relative gaps in (c) are computed as in \eqref{rel_gap}. Each row corresponds to a specific stepsize $\gamma \alpha$, where $\gamma$ ranges from $0.02$ (first row) to $1.0$ (last row).
}
\label{fig:firm_relax} 
\end{figure}

To compare the quality of images $\widehat{w}^{|\Theta|,s,\ell}$ reconstructed from \eqref{basic_model} with $\widehat{w}^{\text{IIR},|\Theta|,s,\ell}$ from the IIR model for $|\Theta| \in \mathcal{N}_{\theta}$, $s \in \mathcal{S}$, and $\ell \in [10]$ obtained for each stepsize scale factor $\gamma \in \Gamma$, each row of Figure \ref{fig:firm_relax} shows the distance from the reconstructed images to the ground-truth, $\|\widehat{w}-\widehat{w}^{\texttt{true}} \|$, along with the relative gap in percentage for the solutions obtained using the stepsize scale factors $\gamma \in \Gamma$ in the given order. The relative gap is computed as

{\footnotesize\begin{align}
100 \times \frac{ \| \widehat{w}^{\texttt{IIR},|\Theta|,s,\ell} - {w}^{\texttt{true}}\| - \| \widehat{w}^{|\Theta|,s,\ell} - {w}^{\texttt{true}}\|}{\| \widehat{w}^{\texttt{IIR},|\Theta|,s,\ell} - {w}^{\texttt{true}}\|}, \ \forall |\Theta| \in \mathcal{N}_{\theta}, s \in \mathcal{S}, \ell \in [10], \label{rel_gap}
\end{align}}
{\color{black}We note that for the noise-free dataset (i.e., $s=0$), the initial step size is fixed at $\eta_1 \gets \alpha$, therefore the results in Figure \ref{fig:firm_relax} remain consistent across the rows. 
}
On the other hand, {\color{black}for the noisy cases (i.e., $s >0$), we report the 20th, 80th percentiles, as well as the average of $\| \widehat{w} - \widehat{w}^{\texttt{true}}\|$ across $10$ independent trials for each $|\Theta| \in \mathcal{N}_{\theta}$. 
We observe that the variance of the results is small. 
}

Figures \ref{fig:firm_relax_1} and \ref{fig:firm_relax_2} demonstrate that for both \(\widehat{w}^{|\Theta|,s,\ell}\) and \(\widehat{w}^{\texttt{IIR},|\Theta|,s,\ell}\), the quality of the images improves as more data becomes available (i.e., as \(|\Theta|\) increases) and degrades as the noise level \(s\) increases, for across all stepsize scale factors \(\gamma \in \Gamma\). Additionally, Figure \ref{fig:firm_relax_3} shows that, for all noise levels and stepsizes tested, \(\widehat{w}^{|\Theta|,s,\ell}\) consistently yields better image quality than \(\widehat{w}^{\texttt{IIR},|\Theta|,s,\ell}\). This benefit is particularly noticeable for the noiseless dataset, highlighting the effectiveness of incorporating multimodality to address the multiple optimal solution issue by leveraging noiseless data points. As the noise level \(s\) increases, the relative gap tends to decrease, although \(\widehat{w}^{|\Theta|,s,\ell}\) remains of better quality, but with a lower relative gap. This suggests that while high noise levels may reduce the impact of the additional information from the multimodality constraint, \(\widehat{w}^{|\Theta|,s,\ell}\) still provides better image quality than the individual model.
Another noteworthy point is that the quality of \(\widehat{w}^{|\Theta|,s,\ell}\) is less affected by small stepsizes than \(\widehat{w}^{\texttt{IIR},|\Theta|,s,\ell}\). Specifically, as the stepsize decreases, the change in image quality is more pronounced in \(\widehat{w}^{\texttt{IIR},|\Theta|,s,\ell}\). This suggests that the solution to \eqref{basic_model} obtained by \texttt{FIRM} exhibits better stability compared to the solution from the IIR model obtained by \texttt{PGD}.

\section{Conclusions}
\label{sec:conclusions}
We proposed \texttt{FIRM}, a federated algorithm that solves a joint inverse constrained optimization problem to reconstruct images from multimodal tomographic datasets generated by multiple agents at different locations, without the need for data centralization. By replacing the costly projection step in \texttt{FedPGD} with simple vector operations, \texttt{FIRM} achieves substantial computational improvements, making it well-suited for federated environments where agents perform local gradient computations on their own data and the server conducts only lightweight computations. We establish a connection between \texttt{FIRM} and the classical inexact QP method and introduce an adaptive step-size rule that guarantees sublinear convergence to an optimal solution at a rate of \(\mathcal{O}(1/\epsilon^2)\), with a condition that improves the rate to \(\mathcal{O}(1/\epsilon)\), matching the rate of \texttt{FedPGD} while requiring significantly less cental computation.

Numerical experiments show that \texttt{FIRM} achieves approximately 52 times faster computational speed compared to \texttt{FedPGD}, while maintaining similar solution quality for the noiseless dataset. For noisy data, \texttt{FIRM} demonstrates its stability by gracefully terminating with a feasible solution, while \texttt{FedPGD} fails to do so even after several days of computation. Additionally, we validate the effectiveness of multimodality in tomographic image reconstruction, showing that \texttt{FIRM} consistently produces higher-quality images than unimodal models across all scenarios, even with fewer or noisier samples, and is less sensitive to variations in the choice of initial step sizes.

In future work, we plan to extend \texttt{FIRM} to low-rank image reconstruction using tensor decomposition, aiming to improve reconstruction quality in ill-posed tomographic problems. Moreover, the connection to first-order QP and AL methods offers opportunities for algorithmic enhancements \cite{lan2013iteration, lan2016iteration, necoara2019complexity, xu2021first, lu2023iteration}. Potential directions include extending \texttt{FIRM} to solve perturbed problems with improved convergence rates, leveraging optimal Nesterov’s method, or developing more effective penalty parameter update rules for the proposed vanilla augmented Lagrangian counterpart. These directions represent exciting prospects for advancing both theory and applications.

\bibliographystyle{siamplain}
\bibliography{references}

\newpage
\appendix

\section{Proof of Lemma \ref{lemm:descent}}
\label{appendix:proof-descent-lemm}
Note that from the smoothness of $Q$, we have
{\begin{align*}
Q(w') & \le Q(w)+\langle \nabla Q(w),w'-w \rangle + \frac{L}{2}\|w'-w\|^2\\
& \le Q(w)+\langle \nabla Q(w),w'-w \rangle + \frac{3}{4\eta}\|w'-w\|^2, \mbox{ since $\eta \le \frac{3}{2L}$}\\
&=Q(w) - \frac{1}{4\eta} \|w'-w\|^2 +\frac{1}{\eta}\langle w-w'-\eta \nabla Q(w), w-w' \rangle , 
\end{align*}}
where the last equality is derived by adding and subtracting $\frac{1}{\eta}\|w'-w\|^2$. \\
If $\langle w-w'-\eta\nabla Q(w), w-w' \rangle \le 0$, the desired result follows. Note that 
{\begin{align*}
 \langle w-w'-\eta\nabla Q(w), w-w' \rangle  =&\left\langle \underbrace{w -\eta \nabla Q(w)}_{=:y}-(w- \eta\nabla Q(w))_+, w - (w- \eta\nabla Q(w) )_+\right\rangle\\
=&\left\langle y - y_+, w - y_+\right\rangle \le  0,
\end{align*}}
where the last inequality holds due to the projection theorem.

\section{Proof of Lemma \ref{lemm:monotone}}\label{proof:monotone}
Let $H$ be the Hessian of the quadratic function $Q$. Since $Q$ is $L$-smooth, we have $\lambda_{\max}(H) \le L$. Since the projection onto a closed convex set is 1-Lipschitz, we have 
{\begin{align*}
\|w_{++} - w_+ \|^2 & \le \|w_{+} - \eta\nabla Q(w_+) - (w - \eta\nabla Q(w)) \|^2\\
& = \|w_{+}-w - \eta H (w_+ - w) \|^2 \mbox{ (since $Q$ is quadratic)}\\
& = \|(I-\eta H) (w_+ - w)\|^2\\
& \le \max\{|\lambda_{\min}(I-\eta H)|, |\lambda_{\max}(I-\eta H)|\}^2\|w_+ - w\|^2\\
& = \max\{|1-\eta \lambda_{\max}(H)|, |1-\eta \lambda_{\min}(H)|\}^2\|w_+ - w\|^2\\
& \le \|w_+ - w\|^2,
\end{align*}}
where the last inequality follows since $\eta \lambda_{\max}(H) \le \eta L \le 2$.

\section{Proof of Lemma \ref{lemm:tech}}
\label{appendix:tech}
(i) Note that the optimal objective value of the following problem gives the tight upper bound on $\|z-z^*\|$ for the case $|h(z) - h(z^*)| \le \delta$:
\begin{equation}\max \{\|z-z^*\|: |\|z\|^2 - \|z^*\|^2 | \le \delta\}.\label{eq:norm-sq:bound}\end{equation}
Note that 
{\begin{align*}
\sqrt{\delta + \|z^*\|^2} + \|z^*\| = \max_z \ & \|z\|+\|z^*\|: \|z\|^2 \le  \eta + \|z^*\|^2
\end{align*}}
is a relaxation to \eqref{eq:norm-sq:bound}, so the desired result follows. This upper bound is indeed tight since it is attained at $z = -\frac{\sqrt{\delta+\|z^*\|}}{\|z^*\|}z^*$. 

(ii) Since $x$ and $y$ are nonnegative, by squaring both sides of the inequality, we get
{\begin{align*}
\sqrt{x+y}+\sqrt{x} \le 2\sqrt{x+\frac{y}{2}}& \Leftrightarrow 2x+ y + 2\sqrt{x+y}\sqrt{x} \le 4(x+\frac{y}{2})\\
 & \Leftrightarrow 2\sqrt{x+y}\sqrt{x} \le 2x+y,
\end{align*}}
which holds by the arithmetic-geometric mean inequality applied to $x+y$ and $x$.

\section{Visualizations of the tomographic multimodal datasets} \label{apx:data}
We present visualizations of some tomographic multimodal datasets used for image reconstruction in our numerical experiments. 
In Figure \ref{fig:experimental_data}, we visualize $b_i \in \mathcal B_{|\Theta|,s}$ for all $i \in [N]$ with $(|\Theta|,s)=(25, 0)$ (first row), $(100,0)$ (second row), and $(100,0.1)$ (last row).
The images in the first row have lower resolutions than those in the second row due to a smaller number of angles. 
The images in the last row appear less clear than those in the second row because of the noise introduced.

\begin{figure}[h!]
\centering
\begin{subfigure}[b]{0.18\textwidth}
\centering    
\includegraphics[width=\textwidth,trim={2cm 0.5cm 2cm 1cm},clip]{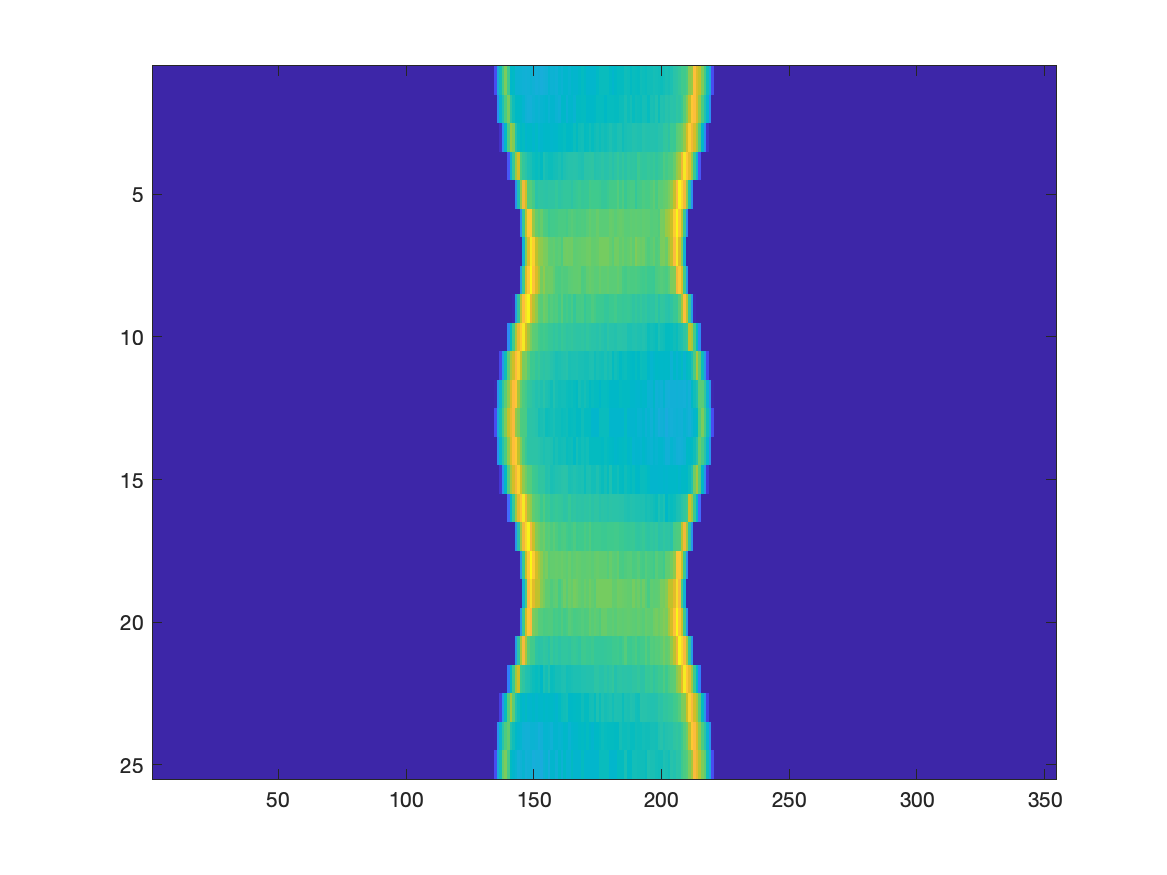}  
\includegraphics[width=\textwidth,trim={2cm 0.5cm 2cm 1cm},clip]{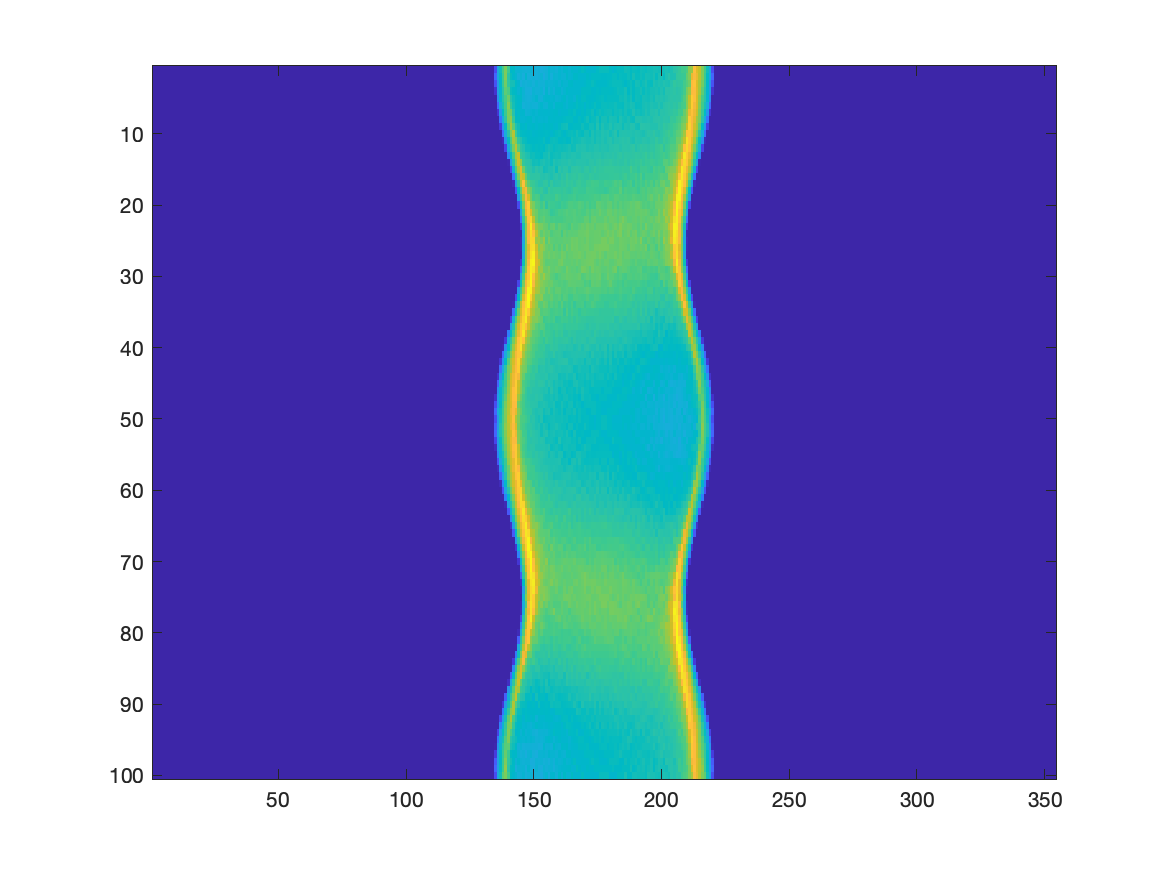}   
\includegraphics[width=\textwidth,trim={2cm 0.5cm 2cm 1cm},clip]{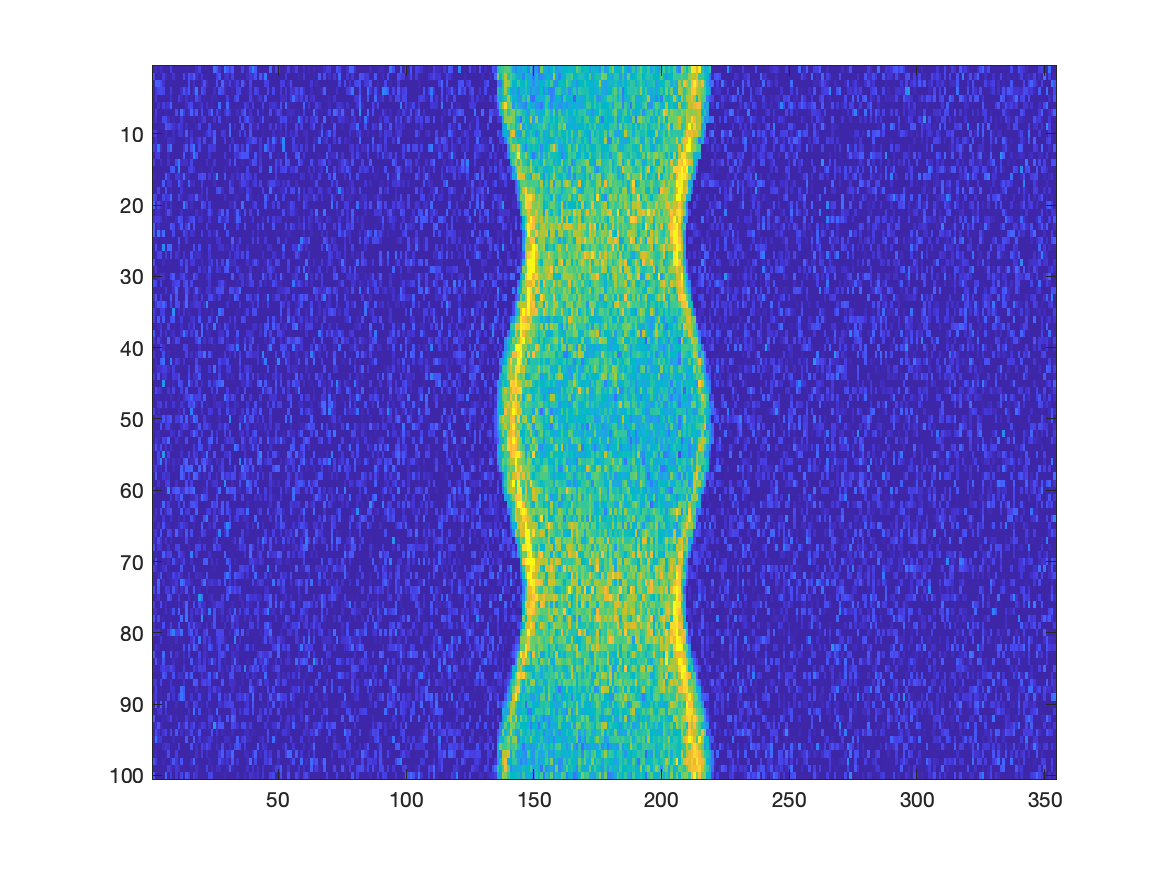}   
\caption{XRF-1} 
\end{subfigure}
\begin{subfigure}[b]{0.18\textwidth}
\centering    
\includegraphics[width=\textwidth,trim={2cm 0.5cm 2cm 1cm},clip]{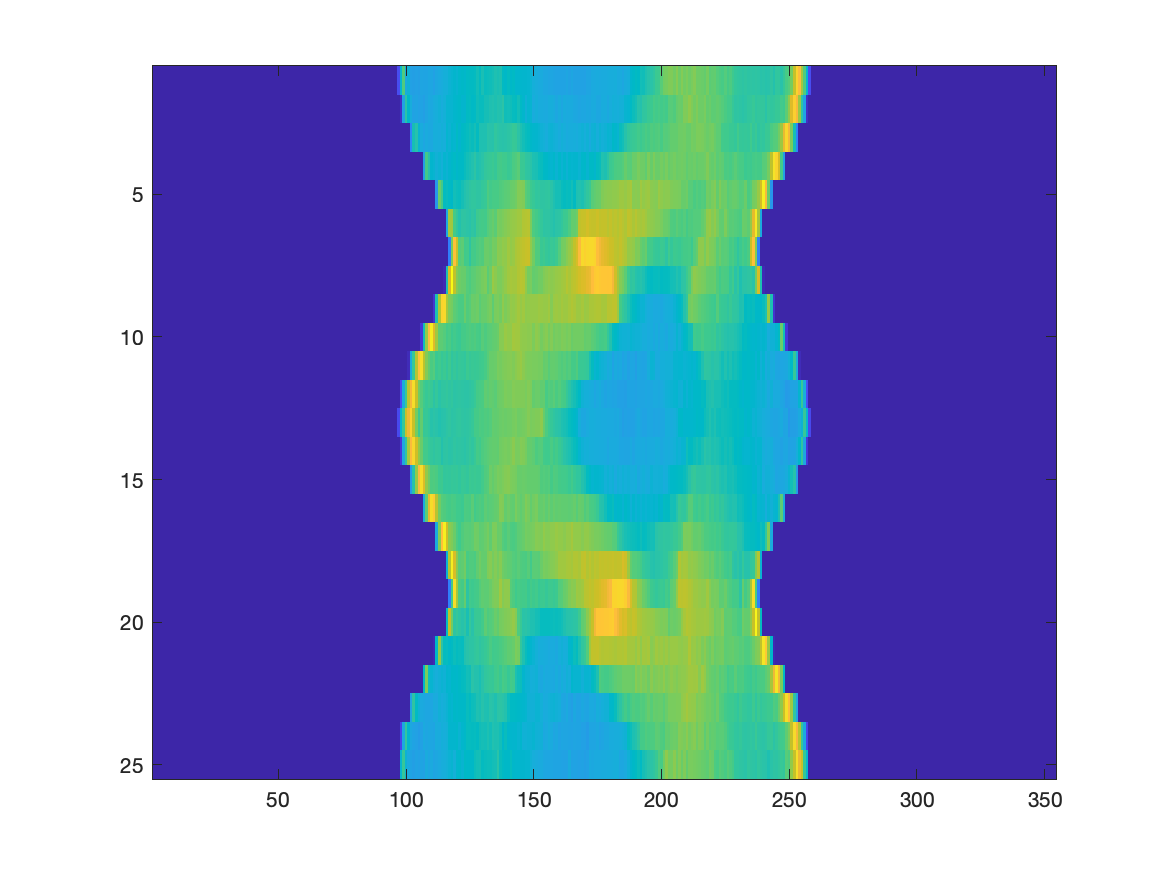}  
\includegraphics[width=\textwidth,trim={2cm 0.5cm 2cm 1cm},clip]{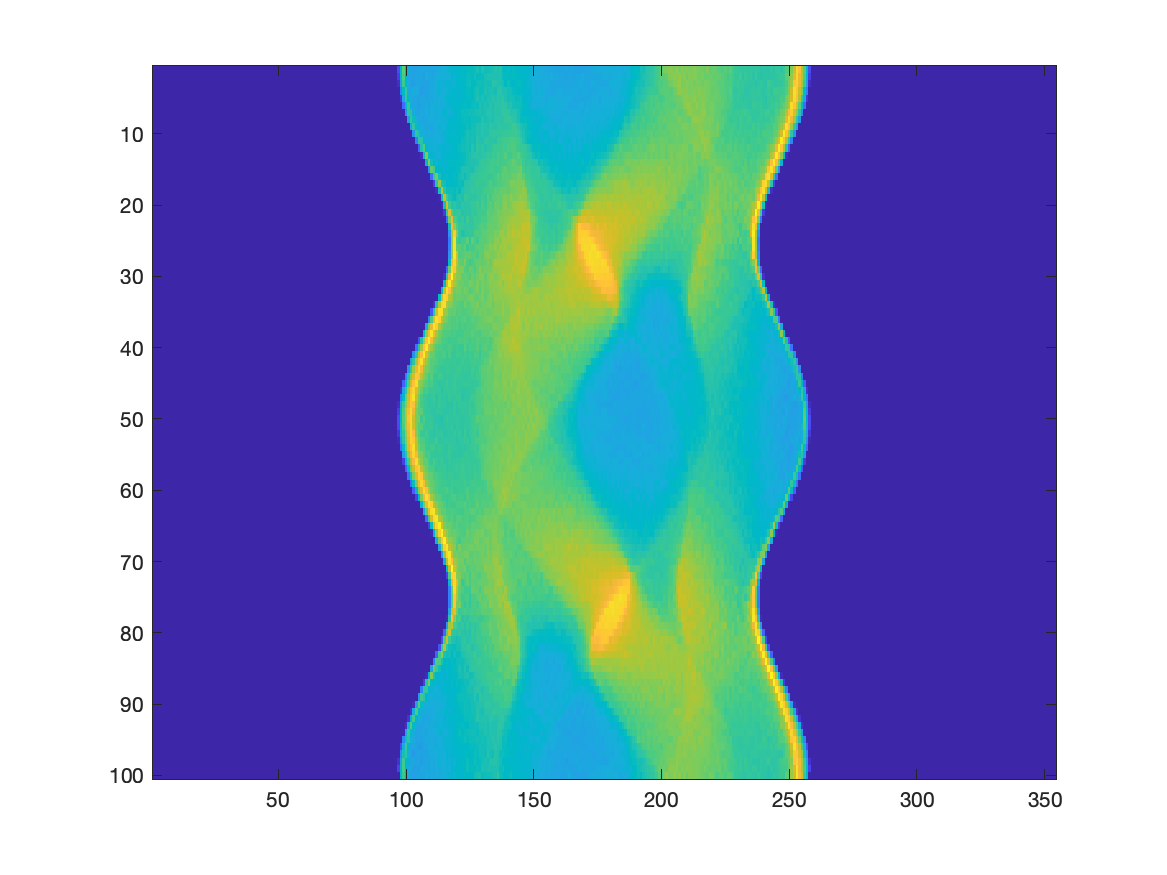}
\includegraphics[width=\textwidth,trim={2cm 0.5cm 2cm 1cm},clip]{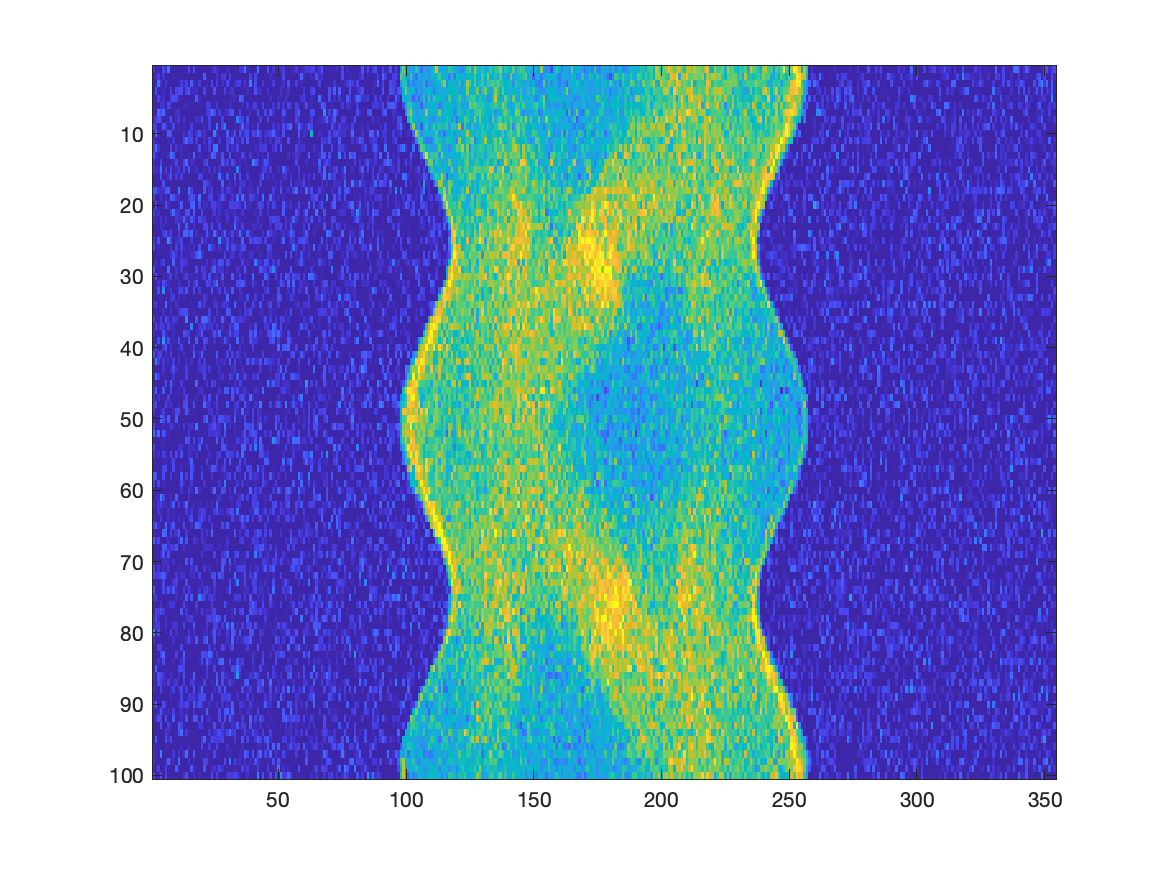}
\caption{XRF-2} 
\end{subfigure}
\begin{subfigure}[b]{0.18\textwidth}
\centering    
\includegraphics[width=\textwidth,trim={2cm 0.5cm 2cm 1cm},clip]{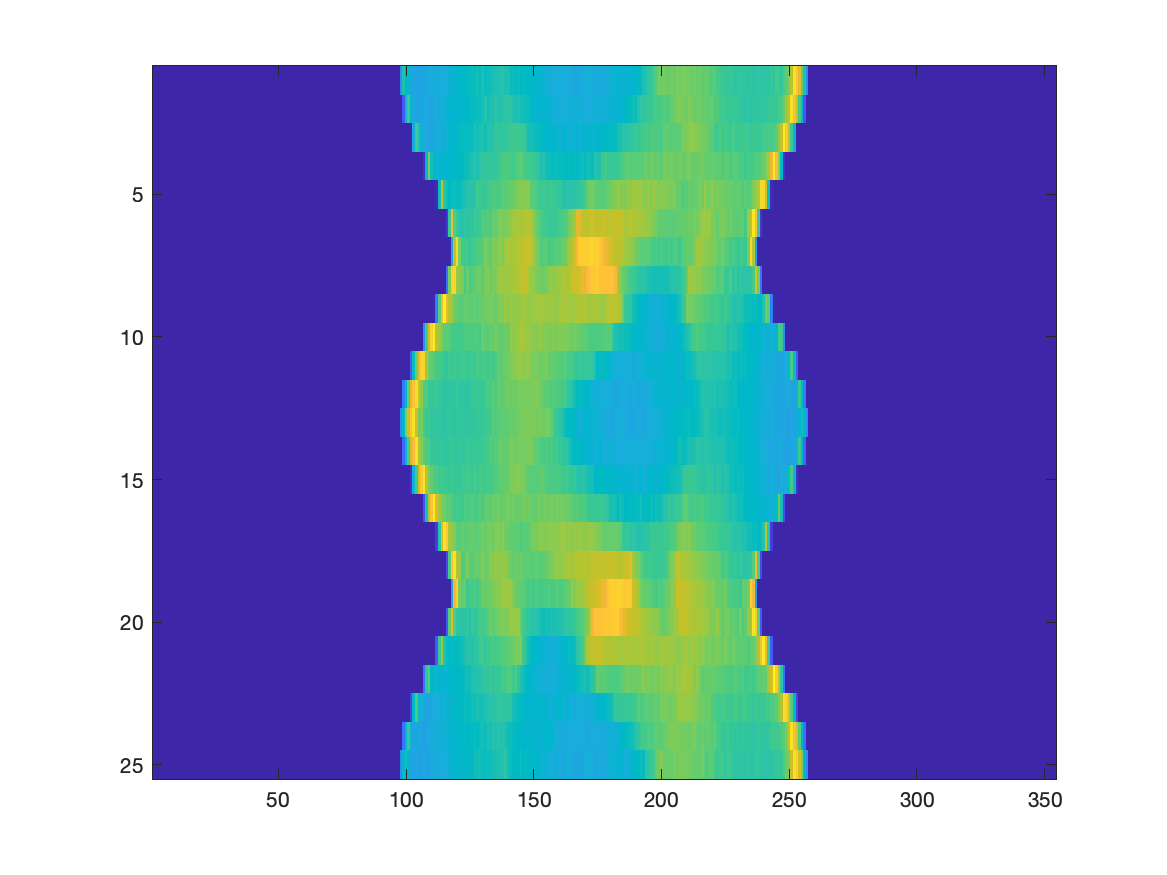}  
\includegraphics[width=\textwidth,trim={2cm 0.5cm 2cm 1cm},clip]{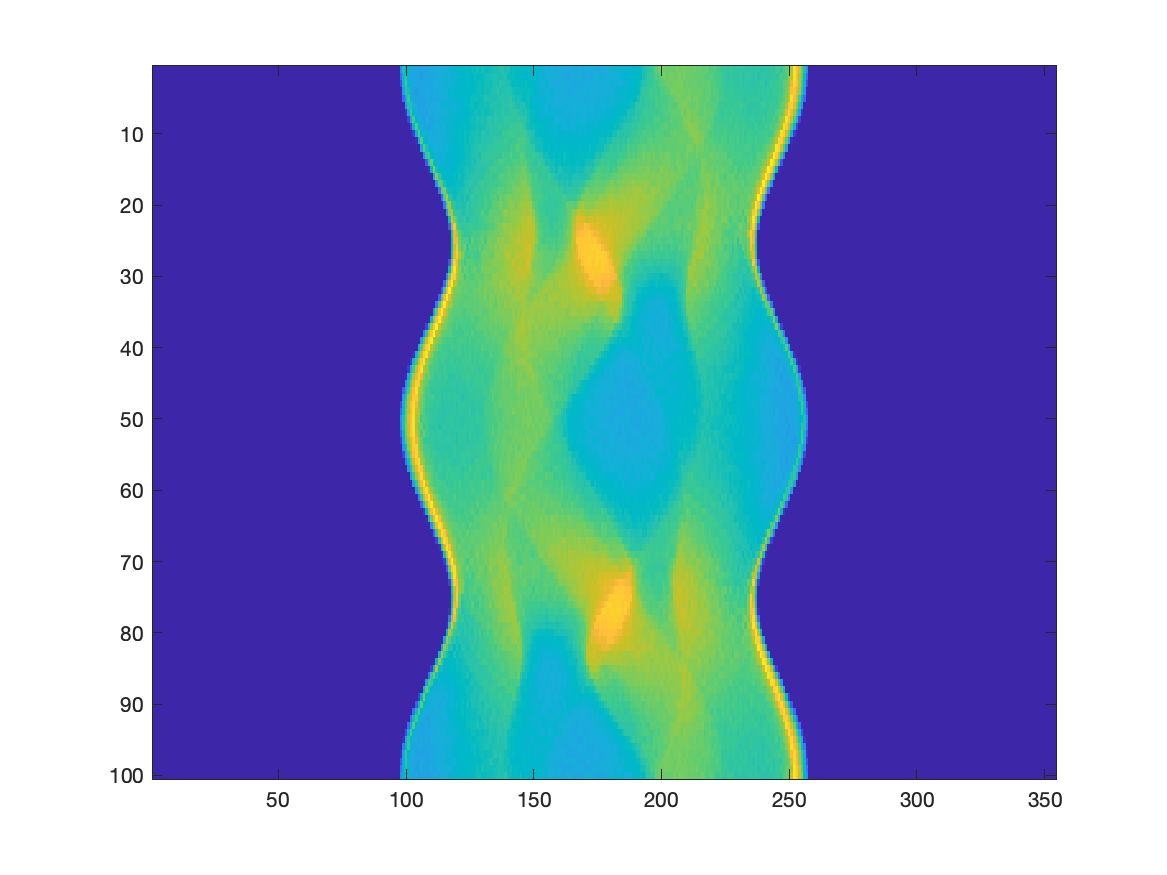}
\includegraphics[width=\textwidth,trim={2cm 0.5cm 2cm 1cm},clip]{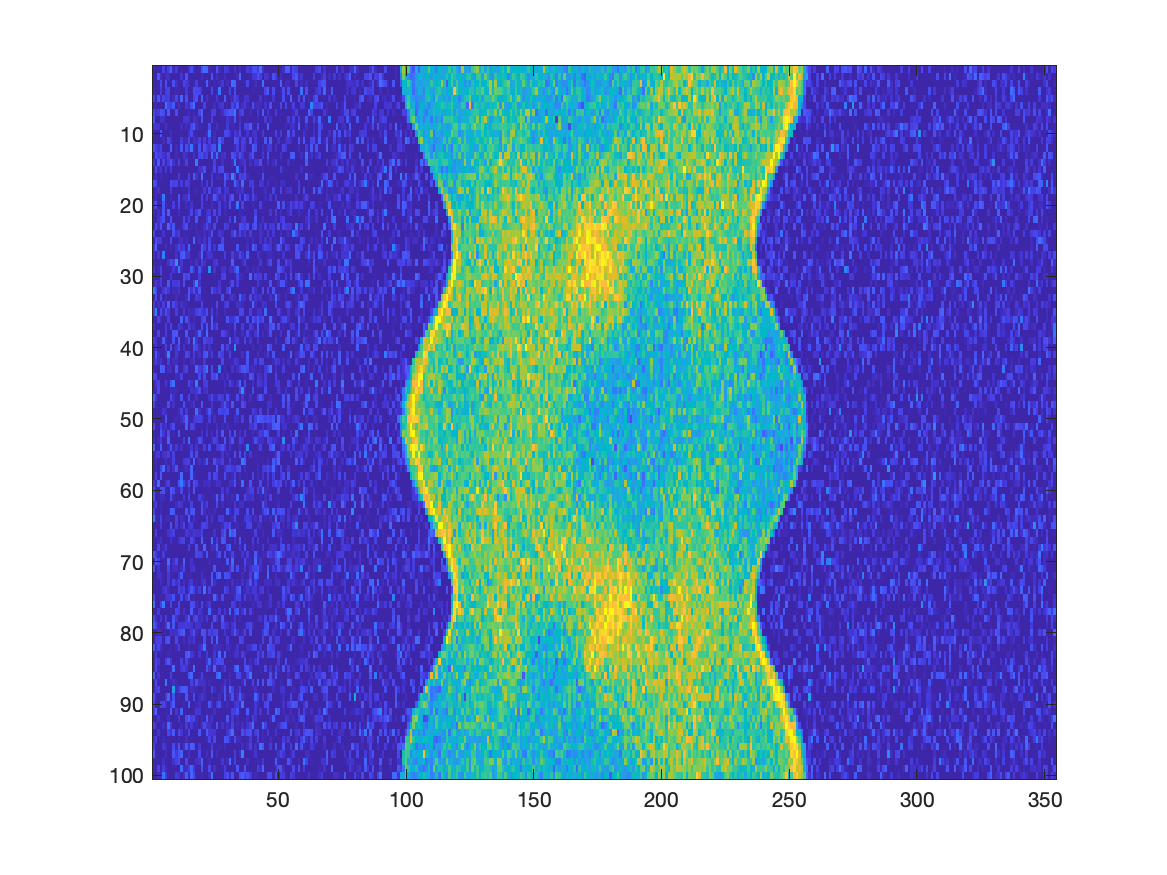}
\caption{XRF-3} 
\end{subfigure}
\begin{subfigure}[b]{0.18\textwidth}
\centering    
\includegraphics[width=\textwidth,trim={2cm 0.5cm 2cm 1cm},clip]{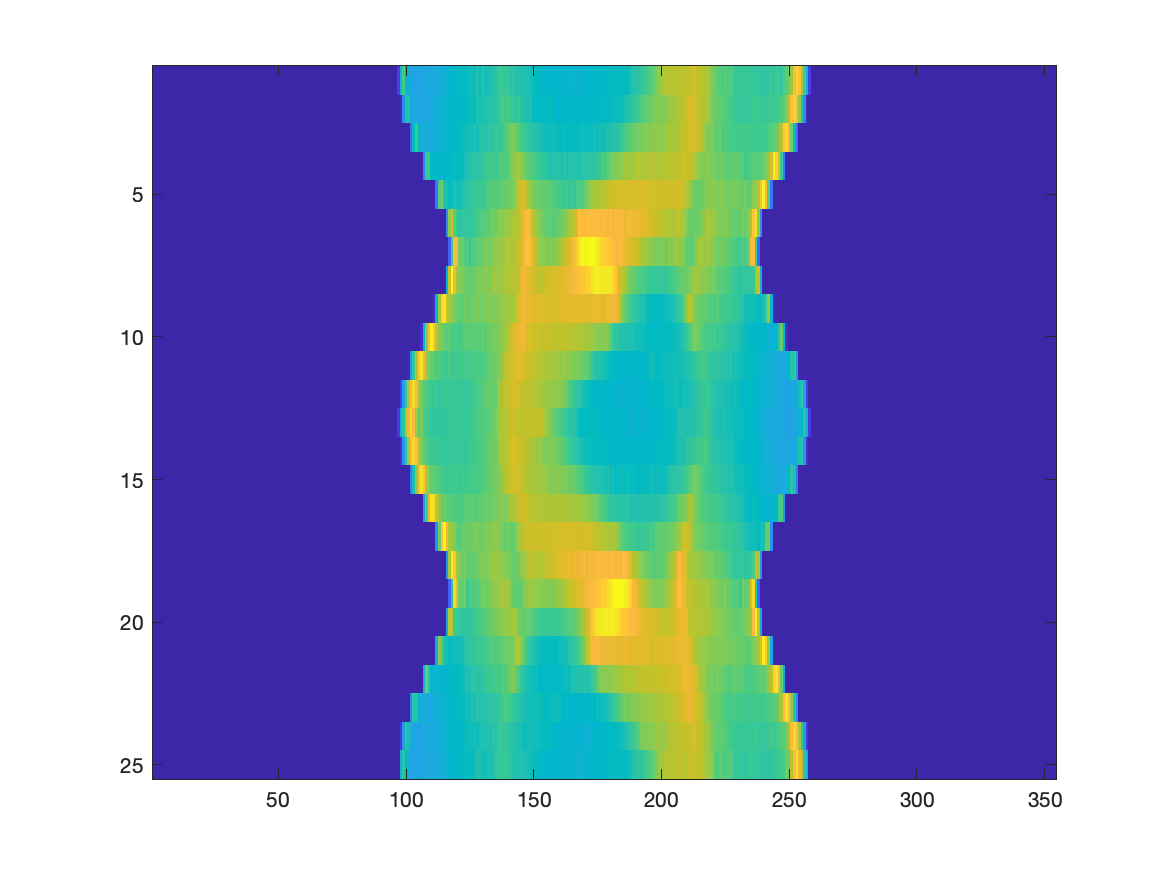}  
\includegraphics[width=\textwidth,trim={2cm 0.5cm 2cm 1cm},clip]{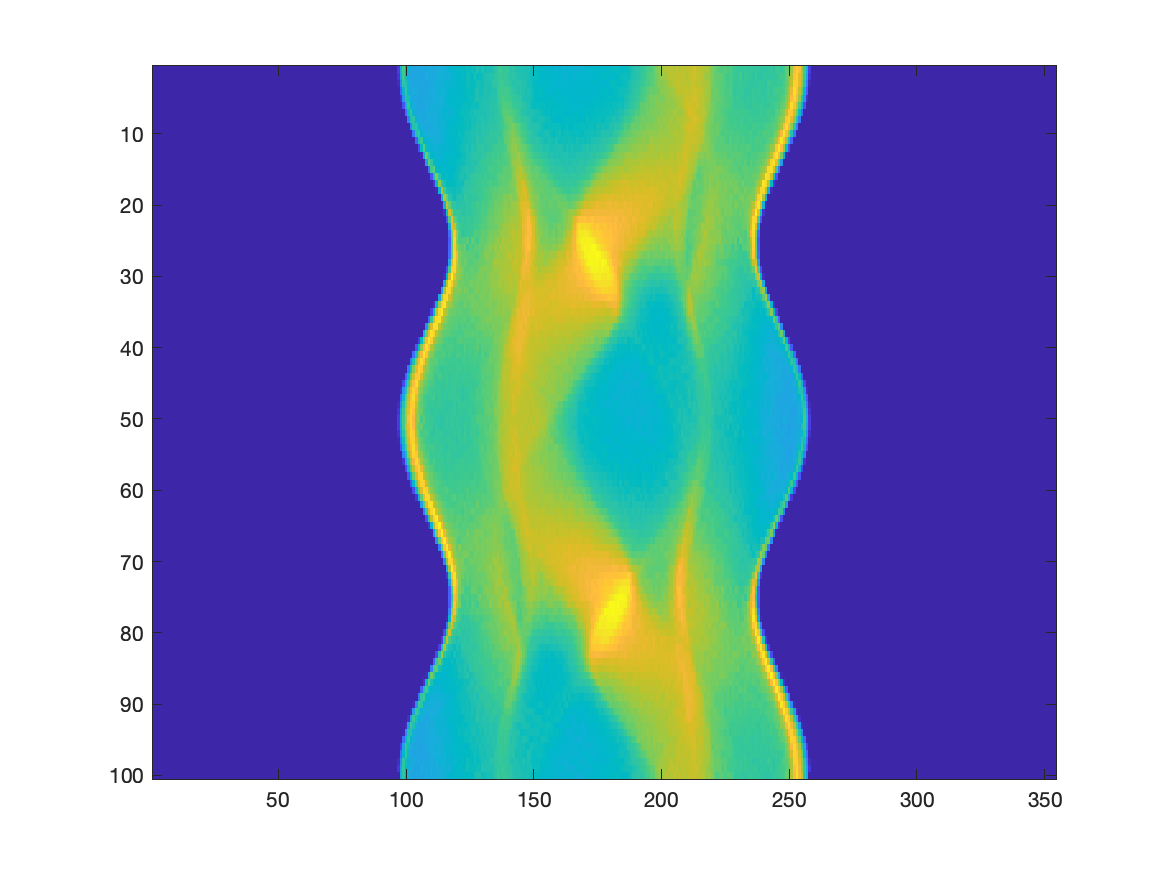}
\includegraphics[width=\textwidth,trim={2cm 0.5cm 2cm 1cm},clip]{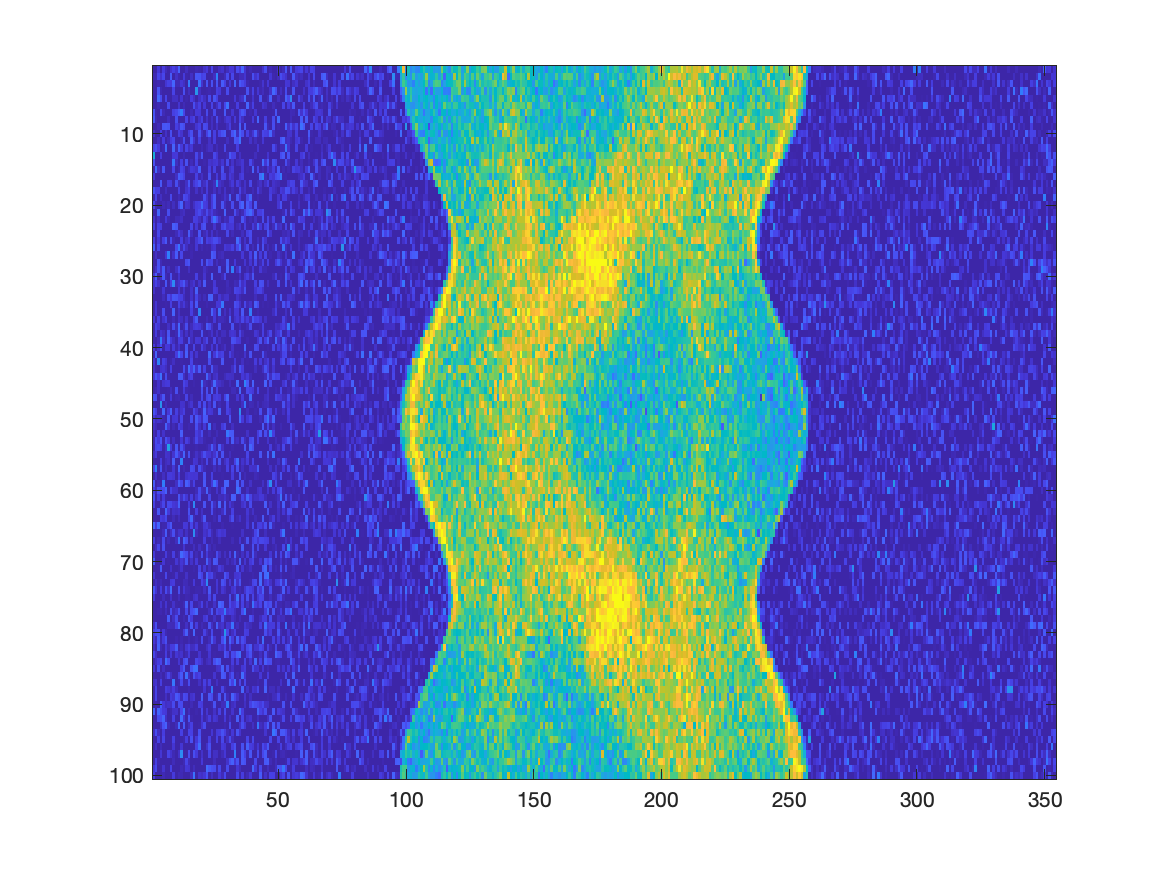}
\caption{XRT} 
\end{subfigure}
\caption{Visualizations of the tomographic multimodal datasets.}
\label{fig:experimental_data} 
\end{figure}

\section{Solution approaches for the IIR Model} \label{apx:iir}
In this section we compare the existing solution approaches for solving the IIR model \eqref{model_individual}, including \texttt{PGD} and the projected least square methods, denoted as \texttt{PLSQR}. First, for fixed $|\Theta| \in \mathcal{N}_{\theta}$ and $s \in \mathcal{S}$, and $i \in [N]$, one \texttt{PGD} step is as follow:
\begin{align}
w^{t+1} \gets   \Big( w^t - \alpha \nabla f_i^{|\Theta|,s} (w^t) \Big)_+,
\end{align}
where $\alpha > 0$ is a constant step size, and
$\nabla f^{|\Theta|,s}_i(w^t) = 2 A^{\top} (Aw^t - b_i )$ is the gradient of the function $f^{|\Theta|, s}_i$ at $w^t$.

On the other hand, \texttt{PLSQR} is the combination of the classical least square method (e.g., \texttt{lsqr} function in Matlab) and the projection, that is the solution from the \texttt{lsqr} is projected onto the non-negative orthant.
We remark that the non-negative least square methods (e.g., \texttt{lsqnonneg} function in Matlab or \texttt{nnls} function in Scipy) could not solve the IIR model \eqref{model_individual} within a reasonable computation time.

In our experiments, we have observed that \texttt{PGD} outperforms \texttt{PLSQR} with respect to the solution quality.
Specifically, we compute $\widehat{w}^{\texttt{PGD},|\Theta|,s}_i$ and $\widehat{w}^{\texttt{PLSQR},|\Theta|,s}_i$ by solving the IIR model via \texttt{PGD} and \texttt{PLSQR}, respectively, for all $|\Theta| \in \{25, 50, 75, 100\}$, $s \in \mathcal{S}$, and $i \in [N]$.
Then we compute the improvement achieved by \texttt{PGD} over \texttt{PLSQR} as follow:

{\begin{align}
100 \times \frac{ \| \widehat{w}^{\texttt{PLSQR},|\Theta|,s}_i - \widehat{w}^{\texttt{true}}_i\| - \| \widehat{w}^{\texttt{PGD},|\Theta|,s}_i - \widehat{w}^{\texttt{true}}_i\|   }{\| \widehat{w}^{\texttt{PLSQR},|\Theta|,s}_i - \widehat{w}^{\texttt{true}}_i\|},
\end{align}}
which are reported in Figure \ref{fig:pgd_lsqr}. The results imply that the IIR model solved by \texttt{PGD} could reconstruct more qualitative images from the tomographic data. 

\begin{figure}[h!]
\centering
\begin{subfigure}[b]{0.24\textwidth}
\centering  
\includegraphics[width=\textwidth,trim={1.0cm 1.0cm 1.0cm 0.25cm},clip]{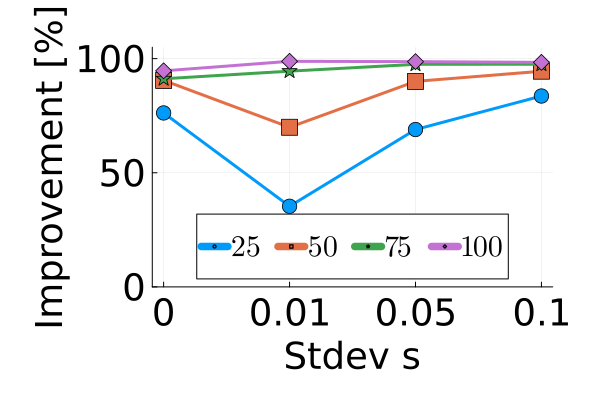} 
\caption{XRF-1} 
\end{subfigure}
\begin{subfigure}[b]{0.24\textwidth}
\centering  
\includegraphics[width=\textwidth,trim={1.0cm 1.0cm 1.0cm 0.25cm},clip]{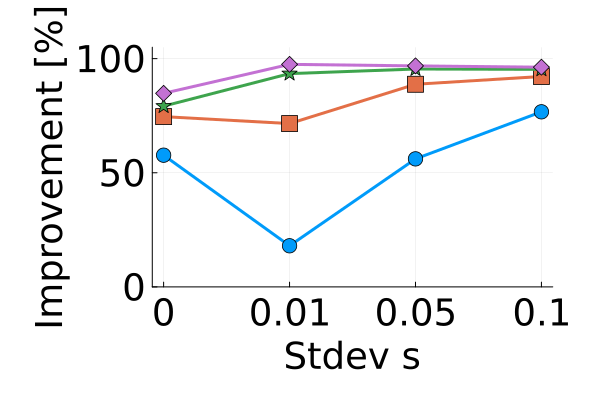} 
\caption{XRF-2} 
\end{subfigure}
\begin{subfigure}[b]{0.24\textwidth}
\centering  
\includegraphics[width=\textwidth,trim={1.0cm 1.0cm 1.0cm 0.25cm},clip]{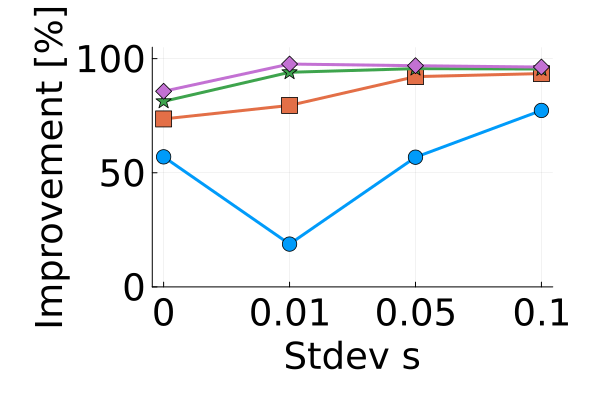} 
\caption{XRF-3} 
\end{subfigure}
\begin{subfigure}[b]{0.24\textwidth}
\centering  
\includegraphics[width=\textwidth,trim={1.0cm 1.0cm 1.0cm 0.25cm},clip]{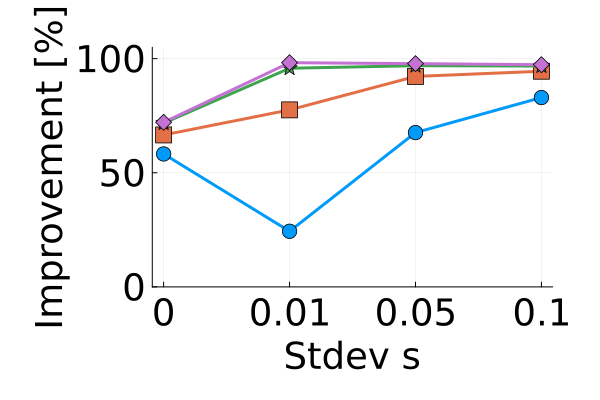} 
\caption{XRT} 
\end{subfigure}
\caption{Improvement achieved by \texttt{PGD} over \texttt{PLSQR}. }
\label{fig:pgd_lsqr} 
\end{figure}

\end{document}